\documentclass[10pt]{article}
\usepackage{amssymb,amsmath,latexcad}
\newtheorem{precor}{{\bf Corollary}}

\newenvironment{cor}{\begin{precor}{\hspace{-0.5
               em}{\bf.\ }}}{\end{precor}}
\newtheorem{precon}{{\bf Conjecture}}

\newtheorem{predefin}{{\bf Definition}}

\newenvironment{defin}[1]{\begin{predefin}{\hspace{-0.5
                   em}{\bf.\ }}{\rm
#1}\hfill{$\spadesuit$}}{\end{predefin}}
\newtheorem{preexm}{{\bf Example}}

\newenvironment{exm}[1]{\begin{preexm}{\hspace{-0.5
                  em}{\bf.\ }}{\rm #1}\hfill{$\clubsuit$}}{\end{preexm}}

\newtheorem{preappl}{{\bf Application}}

\newtheorem{prelem}{{\bf Lemma}}

\newenvironment{lem}{\begin{prelem}{\hspace{-0.5
               em}{\bf.\ }}}{\end{prelem}}
\newtheorem{preproof}{{\bf Proof.\ }}

\newenvironment{proof}[1]{\begin{preproof}{\rm
               #1}\hfill{$\blacksquare$}}{\end{preproof}}
\newtheorem{presproof}{{\bf Sketch of Proof.\ }}

\newtheorem{prethm}{{\bf Theorem}}

\newenvironment{thm}{\begin{prethm}{\hspace{-0.5
               em}{\bf.\ }}}{\end{prethm}}
\newtheorem{prealphthm}{{\bf Theorem}}

\newenvironment{alphthm}{\begin{prealphthm}{\hspace{-0.5
               em}{\bf.\ }}}{\end{prealphthm}}
\newtheorem{prepro}{{\bf Proposition}}

\newenvironment{pro}{\begin{prepro}{\hspace{-0.5
               em}{\bf.\ }}}{\end{prepro}}
\newtheorem{preprb}{{\bf Problem}}

\newenvironment{prb}{\begin{preprb}{\hspace{-0.5
               em}{\bf.\ }}}{\end{preprb}}

\def\conct[#1,#2]{\mbox {${#1} \leftrightarrow {#2}$}}
\def\dconct[#1,#2]{\mbox {${#1} \rightarrow {#2}$}}
\def\deg[#1,#2]{\mbox {$d_{_{#1}}(#2)$}}
\def\mindeg[#1]{\mbox {$\delta_{_{#1}}$}}
\def\maxdeg[#1]{\mbox {$\Delta_{_{#1}}$}}
\def\outdeg[#1,#2]{\mbox {$d_{_{#1}}^{^+}(#2)$}}
\def\minoutdeg[#1]{\mbox {$\delta_{_{#1}}^{^+}$}}
\def\maxoutdeg[#1]{\mbox {$\Delta_{_{#1}}^{^+}$}}
\def\indeg[#1,#2]{\mbox {$d_{_{#1}}^{^-}(#2)$}}
\def\minindeg[#1]{\mbox {$\delta_{_{#1}}^{^-}$}}
\def\maxindeg[#1]{\mbox {$\Delta_{_{#1}}^{^-}$}}
\def\isdef{\mbox {$\ \stackrel{\rm def}{=} \ $}}
\def\dre[#1,#2,#3]{\mbox {${\cal E}_{_{#3}}(#1,#2)$}}
\def\pdre[#1,#2,#3]{\mbox {${\cal P}_{_{#3}}(#1,#2)$}}
\def\var[#1,#2]{\mbox {${\rm Var}_{_{#1}}(#2)$}}
\def\ls[#1]{\mbox {$\xi^{^{#1}}$}}
\def\hom[#1,#2]{\mbox {${\rm Hom}({#1},{#2})$}}
\def\onvhom[#1,#2]{\mbox {${\rm Hom^{v}}(#1,#2)$}}
\def\onehom[#1,#2]{\mbox {${\rm Hom^{e}}(#1,#2)$}}
\def\core[#1]{\mbox {$#1^{^{\bullet}}$}}
\def\cay[#1,#2]{\mbox {${\rm Cay}({#1},{#2})$}}
\def\cays[#1,#2]{\mbox {${\rm Cay_{s}}({#1},{#2})$}}
\def\dirc[#1]{\mbox {$\stackrel{\rightarrow}{C}_{_{#1}}$}}
\def\cycl[#1]{\mbox {${\bf Z}_{_{#1}}$}}

\def\dirp {\mbox {$\stackrel{\rightarrow}{\partial}$}}
\def\rdirp {\mbox {$\stackrel{\leftarrow}{\partial}$}}
\def\simp {\mbox {$\overline{\partial}$}}
\def\dirg {\mbox {$\stackrel{\rightarrow}{\nabla}$}}
\def\dirL {\mbox {$\stackrel{\rightarrow}{\Delta}$}}
\def\simg {\mbox {${\nabla}$}}
\def\siml {\mbox {$\overline{\Delta}$}}
\def\ssimg {\mbox {$\overline{\nabla}$}}

\def\distf {\mbox {$ {\cal O} $}}

\def\odirb {\mbox {$\stackrel{\rightarrow}{E}$}}
\def\idirb {\mbox {$\stackrel{\leftarrow}{E}$}}
\def\sb {\mbox {$\stackrel{\leftrightarrow}{E}$}}
\def\sdg[#1]{\mbox {$\stackrel{\leftrightarrow}{#1}$}}
\def\tr{\mathop{\rm tr}}

\begin{document}
{\footnotesize JCTB {\bf 100} (2010) 390--412}\\
\begin{center}
{\Large \bf  On The Isoperimetric Spectrum of Graphs \\ and Its Approximations}\\
\vspace*{0.5cm}
{\bf Amir Daneshgar\footnote{Correspondence should be addressed to {\tt
daneshgar@sharif.ir}.}
}\\
{\it Department of Mathematical Sciences} \\
{\it Sharif University of Technology} \\
{\it P.O. Box {\rm 11155--9415}, Tehran, Iran}\\
{\tt daneshgar@sharif.ir}\\ \ \\
{\bf Hossein Hajiabolhassan}\\
{\it Department of Mathematical Sciences}\\
{\it Shahid Beheshti University, G.C.}\\
{\it P.O. Box {\rm 19839-63113}, Tehran, Iran}\\
{\tt hhaji@sbu.ac.ir}\\ \ \\
{\bf Ramin Javadi}\\
{\it Department of Mathematical Sciences} \\
{\it Sharif University of Technology} \\
{\it P.O. Box {\rm 11155--9415}, Tehran, Iran}\\
{\tt rjavadi@mehr.sharif.ir}\\ \ \\
\end{center}
\begin{abstract}
\noindent In this paper\footnote{This article is a revised version of \cite{DAHAARXIV} distributed on arXiv.org (1'st, Jan. 2008).}
we consider higher isoperimetric numbers of a (finite directed) graph.
In this regard we focus on the $n$th mean isoperimetric
constant of a directed graph as the minimum of the mean outgoing
normalized flows from a given set of $n$ disjoint subsets of the
vertex set of the graph.
We show that the second mean
isoperimetric constant in this general setting, coincides with (the
mean version of) the classical Cheeger constant of the graph, while
for the rest of the spectrum we show that there is a fundamental
difference between the $n$th isoperimetric constant and the number
obtained by taking the minimum over all $n$-partitions. In this
direction, we show that our definition is the correct one in the sense that it satisfies a
Federer-Fleming-type theorem, and we also  define and present examples for the concept of a {\it supergeometric graph} as a graph
whose mean isoperimetric constants are attained on partitions at all levels.\\
Moreover, considering the ${\bf NP}$-completeness of the isoperimetric problem on graphs,
we address ourselves to the approximation problem where we prove general spectral inequalities
that give rise to a general Cheeger-type inequality as well. On the other hand, we also consider some algorithmic aspects
of the problem where we show connections to orthogonal representations of graphs
and following J.~Malik and J.~Shi ($2000$) we study the close relationships to the well-known
$k$-means algorithm and normalized cuts method.\\
\begin{itemize}
\item[]{{\footnotesize {\bf Key words:}\  Isoperimetric number, Connectivity, Graph, Markov chain.}}
\item[]{ {\footnotesize {\bf Subject classification:}\ 05C40, 60J10.}}
\end{itemize}
\end{abstract}
\section{Introduction}
\subsection{Objectives and main results}
The chapter on isoperimetric numbers and Cheeger-type inequalities
is a classic in geometric analysis as well as spectral graph theory and has been considered from many
different aspects and points of view
\cite{BAN80,BUS92,CHA84,DIA88,GY03,LUB94,MT06}. The study of such
concepts in the discrete case, although more recent, has also been a
center of attention mainly because of its many diverse connections
to important problems of the century, both
applied and theoretical in nature (e.g. see \cite{ALFI96,BHT00,CH97,COL98,CV1,CV2,DIA03,GOD,HAR04,HT04,Le91,MOH89,SAL?,SAL97,SIMA00}).\\
Let us recall (e.g. see \cite{SAL97}) the definition of the
classical Cheeger constant of a Markov chain\footnote{By abuse
of language we only refer to the kernel instead of the stochastic
process itself.} $K$ on a directed base-graph
$G=(V(G),E(G))$, with a nowherezero stationary distribution $\pi$, as
$$\varsigma(K,\pi) \isdef \displaystyle{\min_{_{\pi(Q) \leq 1/2}}}\ \frac{\partial(Q)}{\pi(Q)}=\displaystyle{\min_{_{Q \subseteq V(G)}}}\ \max \left \{\frac{\partial(Q)}{\pi(Q)},\frac{\partial(Q)}{\pi(Q^c)} \right \},$$
where
$$\partial(Q) \isdef \displaystyle{\sum_{_{u \in Q\ \& \ v \not \in Q}}} K(u,v)\pi(u),$$
and the not so common mean version as follows
\begin{equation}
\iota(K,\pi) \isdef \displaystyle{\min_{_{Q \subseteq V(G)}}}\
\frac{\partial(Q)}{2\pi(Q)(1-\pi(Q))}.
\end{equation}
Our main objective in this paper is to analyze the {\it connectivity} of a (directed) finite graph
through its isoperimetric spectrum and to consider the computational aspects of this problem.
In this regard, we will show that the natural generalization of the classical definition to the $n$th
level for $n > 2$ does not work if one takes the minimum over $n$-partitions of the vertex set and we propose a
correct definition in the sense that it satisfies a coarea formula. We also show that this fundamental difference
not only is important in the approximation problem of the isoperimetric spectrum but also can be quite discriminating
in an algorithmic approach to applications.\\
In what follows we try to present a short overview of our approach
in this article.  Firstly, it should be noted that our main strategy
is to transfer the problem to the symmetric base graph and then use
the machinery that is already available through the theory of
reversible Markov chains. Our basic {\it symmetrization} approach is
adopted from some classical methods as presented in \cite{FIL91}.
In this regard, in Section~\ref{PRE}, on the one hand, we have tried to present this well-known setup
in a unified and concise way accessible to graph theorist with an emphasis on concepts related to connectivity and flows on the base graph,
where we have tried to keep our notations classic enough to be most natural for all communities involved.
On the other hand, in this section we show that the standard symmetrization process that translates the connectivity parameters
of a given (in general directed) graph to the connectivity of its symmetrized undirected base graph is natural in the sense that
the stationary distribution and the total flow are preserved in this change of base (see Equation~(\ref{LEMMA2A}) and Lemma~\ref{LEMGER}($a$)).\\
Secondly, it also ought to be noted that computing the isoperimetric
numbers of a graph is known to be a computationally hard problem
(e.g. see \cite{MOH89} for the classic version. For the mean version see \cite{VEMPALANP} or \cite{SIMA00} where an ${\bf NP}$-completeness result
attributed to Papadimitrou~$(1997)$ is presented), when the eigenvalues
and eigenfunctions of a finite graph are quite close at hand and can
be computed effectively (essentially in polynomial time) through
well-known methods of linear algebra.\\
In the sequel, we concentrate on the  {\it mean}
version of the isoperimetric number, and introduce its extensions to
higher indices, as a set of constants called the (mean) {\it isoperimetric
spectrum}, in juxtaposition to the classical spectrum consisting of
the Laplacian eigenvalues. \\
It is instructive to note that although the idea of using the maximum version of the higher isoperimetric numbers as
$$\varsigma_{_{n}}(G,K) \isdef
\displaystyle{\min_{_{ \{Q_{_{i}}\}^{^{n}}_{_{1}}  }} } \ \displaystyle{\max_{1 \leq i \leq n}} \left (
 \frac{\partial (Q_{_{i}})}{\pi(Q_{_{i}})} \right ),$$
can be traced back into some texts as \cite{BUS92}, it seems that the subtle problem of choosing the suitable class of  subsets $\{Q_{_{i}}\}^{^{n}}_{_{1}} $ has not been discussed in detail.  \\
The main reason for our shift of interest
toward the mean version are manifold. On the one hand, we must note that most of our results are correct for both maximum and mean versions,
however usually the proofs for the mean version are more involved. On the other hand,  in our opinion, the {\it
mean spectrum} of the Laplacian whose $n$th element is the
(arithmetic) mean of the first $n$ eigenvalues of the Laplacian
operator, seems to be much more well-behaved than the classical
spectrum because of the smoothing property of the mean operator
(e.g. see \cite{HY95} for the spectral approximations and a perturbation analysis). This may, in a
way, present a fair chance of a better study/approximation of the
spectrum and, in this regard, the generalized mean version of the
isoperimetric constant plays the
central role as the most natural $L^{1}$ counterpart.\\
Another important aspect of considering the mean version is the fact
that it can be traced back into some important applications as clustering (e.g. see \cite{VEMPALANP,SIMA00} among many other references and Section~\ref{KMEANS})
and as far as we could verify, it presents the most
natural applied framework to generalize the isoperimetric constant. In this setup, what
is in our opinion a bit of a {\it surprise}, is that the new
definition seems to be well defined (e.g. in the sense that it
satisfies a generalized co-area formula) only when it is defined as
the minimum over {\it disjoint} subsets of the space (which does not
necessarily constitute a {\it partition}). The difference between
the two definitions based on taking the minimum over {\it
partitions} or {\it disjoint} sets, although disguised in the case
of the classical Cheeger constant (i.e. when we deal with
$2$-partitions; also, see Proposition~\ref{INEQCOR}), seems to be inherently nontrivial in general for both of mean and maximum versions.
 In Section~\ref{ISOSPEC} we introduce and
investigate some basic properties of the generalized mean isoperimetric numbers,
where in Section~\ref{EXAMPLES}, we concentrate on some examples and special cases (specifically, in relation to Theorem~\ref{IOTAEQGAMMA} and the proceeding paragraph).
In Section~\ref{GEOGRAPH} we define the concept of a {\it supergeometric} graph
as a graph for which all parameters involved are equivalent, and we provide some examples of such graphs.
 We believe that supergeometric graphs possess interesting properties that ought to be investigated in future research. \\
 Naturally, pursuing this line of thought, we
analyze the mean isoperimetric spectrum, both from analytic and graph
theoretic points of view, and we prove a Federer-Fleming-type
theorem (Section~\ref{COAREA}) as well as  Cheeger-type inequalities connecting these
parameters and the classical spectrum of eigenvalues in different
levels (Section~\ref{SECGENCHEEGER}). Also, as a byproduct, it is shown that generalized Cheeger
inequalities at the $n$th level seem to be strongly related to the
concept of a {\it nodal domain} (e.g. see \cite{JNT01,ZHU} for the background).\\
Based on the fact that the isoperimetric problem is computationally
a hard  problem,
 in Section~\ref{KMEANS} we concentrate on some algorithmic
considerations related to computation of the isoperimetric spectrum where we study
the close relationships to the well-known $k$-means algorithm.
\section{Preliminaries}\label{PRE}
 In this section we go through some basic definitions and facts
that will be used later. In what follows ${\mathbb R}$ and ${\mathbb R}^+$
are the sets of real  and nonnegative real numbers, respectively,
and for any real number $x \in {\mathbb R}$ we define
$$(x)^+ \isdef \left \{ \begin{array}{ll}
 x & x > 0 \\
 0 & x \leq 0. \end{array}\right. $$
For an $n$-list of real numbers (repetition is allowed) as
$(\zeta_{_{1}},\zeta_{_{2}},\ldots ,\zeta_{_{n}})$, the mean
$n$-list is denoted by
$(\overline{\zeta}_{_{1}},\overline{\zeta}_{_{2}},\ldots
,\overline{\zeta}_{_{n}})$, where
$$\overline{\zeta}_{_{k}} \isdef \frac{1}{k}\ \displaystyle{\sum^{k}_{i=1}}\ \zeta_{_{i}} .$$
Hereafter, we adopt the notation ${\cal I}^{^k}_{_{n}} \isdef
\{k,k+1,\ldots,n\}$. Also, ${\cal I}_{_{n}} \isdef {\cal
I}^{^1}_{_{n}}$.
\subsection{Function spaces}
If $X$ is a set then ${\cal F}^d(X)$ stands for the set of all real
functions $f : X \longrightarrow {\mathbb R}^d$, and also we define ${\cal F}(X) \isdef {\cal F}^1(X)$.
Similarly, ${\cal F}^+(X) \isdef \{f \ | \ f : X \longrightarrow {\mathbb R}^+\}$. Also, for a positive and nowherezero
weight function $\omega : X \longrightarrow {\mathbb R}^+-\{0\}$ we define the inner product $\langle.,.\rangle_{_{\omega}}$ and the norm
$\| . \|_{_{p,\omega}}$ on ${\cal F}(X)$ as
$$\langle f,g\rangle_{_{\omega}} \isdef \displaystyle{\sum_{x \in X}} f(x)g(x)\omega(x), \quad
\| f \|_{_{p,\omega}} \isdef \left ( \displaystyle{\sum_{x \in X}} \
|f(x)|^p \ \omega(x) \right )^{\frac{1}{p}},$$ respectively, where
we usually use the subscript $\omega$ to refer to the product
structure (e.g. ${\cal F}_{_{\omega}}(X)$). Two functions $f,g \in
{\cal F}_{_{\omega}}(X)$ are said to be orthogonal with respect to
$\omega$,
i.e. $f \perp_{_{\omega}} g$, whenever $\langle f,g\rangle_{_{\omega}}=0$.\\
 For any $f \in {\cal F}(X)$,  ${\rm supp}(f)$
stands for the set $\{v \in V(G) \ \ | \ \ f(v) \not = 0 \}$.
Also, for any subset $A \subseteq X$ the restriction of $f$ to
$A$ is denoted by $f|_{_{A}}$, i.e.,
$$f|_{_A}(x) \isdef \left \{ \begin{array}{ll}
f(x) & x \in A  \\
 0 & x \not \in A. \end{array}\right. $$
 The characteristic function of a subset $A \subseteq X$ is denoted by $\chi_{_{A}} \isdef 1|_{_{A}}$
 when $X$ is clear from the context.\\
Moreover, for any real function $f$, the functions $f^+ $ and $f^- $
stand for the positive and negative parts of $f$, respectively; and
consequently,
$$f=f^+ - f^- \quad {\rm and} \quad |f|=f^+ + f^-.$$
 For any two functions (or vectors) $f,g$, we write $f \leq g$ if
$$\forall\ v \in V(G)  \quad f(v) \leq g(v).$$
Also, we write $f < g$ if $f \leq g$ and $f \not = g$.
\subsection{Graphs and kernels}
The main objective of this section is to introduce a common language
of graphs and kernels that is accessible to both graph theorists and
experts in functional analysis and also benefits from all aspects of
the two points of view.\\
 Throughout the paper, a {\it graph} $G=(V(G),E(G))$ is
always assumed to be a finite directed  graph (possibly with loops
and without multiple edges), where $E(G) \subseteq V(G) \times
V(G)$. Similarly, an {\it undirected graph}
$\overline{G}=(V(\overline{G}),E(\overline{G}))$ is
 a finite set $V(\overline{G})$ along with a set of undirected edges $E(\overline{G})$,
 each element of which is a subset of $V(\overline{G})$ whose size is less than or equal to
 $2$.
When it is clear from the context, by abuse of notation, we use the
same symbol $uv$ both for the directed edge $(u,v) \in V(G) \times
V(G)$ of a directed graph and also for a simple edge $\{u,v\}$ of an
undirected  graph.\\
For a given graph $G$, we use the natural notation, $\sdg[G]$, for
its symmetric directed base graph i.e. $V(\sdg[G]) \isdef V(G)$ and
$$ (uv \in E(\sdg[G]) \ {\rm and} \ vu \in E(\sdg[G])) \Leftrightarrow (uv \in E(G) \ {\rm or}\ vu \in
E(G)).$$
 Moreover, for a given graph $G$, $\overline{G}$ stands for its symmetric undirected base
graph i.e. $V(\overline{G}) \isdef V(G)$ and
$$ uv \in E(\overline{G}) \Leftrightarrow (uv \in E(G) \ {\rm or}\ vu \in
E(G)).$$ Note that for an undirected graph
$\overline{G}=(V(\overline{G}),E(\overline{G}))$ we may think of
any simple edge $uv$ as a subset $\{u,v\} \subseteq
V(\overline{G})$. With this interpretation
$\sdg[G]=(V(\sdg[G]),E(\sdg[G]))$ is a directed graph obtained
by replacing any simple edge $uv \in E(\overline{G})$ with two
directed edges $uv \in E(\sdg[G])$ and $vu \in E(\sdg[G])$. Note
that there is a one-to-one correspondence between undirected
graphs and symmetric directed graphs, where the undirected
presentation can be interpreted as a more compact version
of expressing the same data. \\
Given any $n \times n$ matrix  $K$ whose rows and columns are
indexed by the elements of an $n$-set $V$, in general, one can
construct a graph $G_{_{K}}=(V,E)$ where
$$uv \in E \quad \Leftrightarrow \quad K(u,v) \not = 0.$$
Then, it is clear that from this point of view, the concept of a
{\it weighted graph} contains the same data as the concept of a
matrix, and moreover, symmetric graphs as well as undirected graphs
correspond to the concept of  symmetric matrices.\\ \ \\
{\bf Notational assumption:} Throughout the paper an {\it overlined} notation
is usually adopted to refer to a symmetrization process or taking arithmetic means
applied to the original concept.\\  \ \\
For two given subsets $X,Y$ of
$V(G)$ we define
$$\stackrel{\rightarrow}{E}(X,Y) \isdef \{uv \in E(G) \ \  | \ \ u \in X \ \& \ v \in Y\}.$$
 Also, for a subset $Q \subset V(G)$ we define
$$\odirb(Q) \isdef \stackrel{\rightarrow}{E}(Q,Q^c), \quad  \idirb(Q) \isdef \stackrel{\rightarrow}{E}(Q^c,Q),$$
and
$$\sb(Q) \isdef \odirb(Q) \cup \idirb(Q).$$
Hereafter, ${\sf K}_{_{t}}$ and ${\sf K}_{_{r,s}}$ stand for the (simple) complete  graph on $t$ vertices and the complete bipartite graph
on two parts of sizes $r$ and $s$, respectively.
\subsection{Markov kernels and the energy space}
Our major objective in this section is to present a standard symmetrization process
as well as some basic facts about the theory of
finite Markov chains with a graph theoretic emphasis (e.g. for more
on this see \cite{ALFI96,JGT,FIL91,SAL97}).
Particularly, we will show that the basic parameters used in this article as the weight functions i.e. the stationary distribution and the
corresponding natural flow, are preserved in this setup, and consequently, one may talk about the connectivity parameters that are computable
from the symmetric model of the corresponding undirected base graph. In this direction, we have tried to use the {\it overlined} notations
for the parameters  and concepts related to the undirected base graphs where we have used the {\it arrowed} notations for the general directed case.
We will elaborate on the details of this notational assumptions later in this section.\\
Hereafter, given a graph $G$,
we assume that $K$ is the kernel of a Markov chain on this graph and
$\pi$ is a nowherezero stationary distribution, i.e. $\pi K =\pi$
and $\pi(v)
\not = 0$ for all $v \in V(G)$.\\
In this setting, $\phi(u,v) \isdef K(u,v) \pi(u)$ defines a
nowherezero flow on $G$. Also, for any two disjoint sets  $X,Y
\subseteq V(G)$ we define
$$\pi(X) \isdef \displaystyle{  \sum_{_{u \in X}}  }\ \pi(u),
\quad \dirp_{_{\phi}}(X,Y) \isdef \sum_{_{uv \in
\stackrel{\rightarrow}{E}(X,Y)}} \phi(u,v),\quad
\dirp_{_{\phi}}(X) \isdef \dirp_{_{\phi}}(X,X^c),$$ and
$\rdirp_{_{\phi}}(X,Y)$ and $\rdirp_{_{\phi}}(X)$ analogously.
 Note that since $\phi$ is a flow, for every nonvoid subset $Q
\subseteq V(G)$, we have
$$\dirp_{_{\phi}}(Q)=\rdirp_{_{\phi}}(Q)=\dirp_{_{\phi}}(Q^c).$$
By abuse of notation, we may write $\dirp(Q)$ for
simplicity, if the flow is clear from the context.
 Within the same
\footnote{Note that for Riemannian manifolds $\partial
(Q)=Vol(Boundary(Q))$ has the same property as a trivial flow with
only a nonzero component to the complement.}
setup, for the symmetric graph $\overline{G}$, we consider the
kernel
$$\overline{K}(u,v) \isdef \frac{1}{2} (K+K^*)
=\frac{1}{2} \left ( K(u,v)+ \frac{K(v,u)\pi(v)}{\pi(u)}  \right),$$
with the same stationary distribution $\pi$ inducing the flow
$$\overline{\phi}(u,v) \isdef  \overline{K}(u,v) \pi(u)= \frac{1}{2} (\phi(u,v)+\phi(v,u))$$
on $\overline{G}$. We can also define
$\simp_{_{\overline{\phi}}}(Q)$ similarly. Note that
\begin{equation}\label{LEMMA2A}
\simp_{_{\overline{\phi}}}(Q) \isdef \sum_{\genfrac{}{}{0pt}{}{uv\in
E(\overline{G})}{u\in Q, v\in Q^c}} \overline{\phi}(u,v) =\frac{1}{2}(\dirp_{_{\phi}}(Q)+\rdirp_{_{\phi}}(Q))=\dirp_{_{\phi}}(Q).
\end{equation}
This shows that for a given graph $G$, the outgoing flow from a subset $Q \subset V(G)$ is
equal to the outgoing flow from $Q \subset V(\overline{G})$ in the symmetrized model, which
justifies our transformation method from the directed
case to the symmetric case, when dealing with connectivity parameters of graphs
in terms of the corresponding flows. We will also prove a generalization of this fact in
Lemma~\ref{LEMGER}(a).\\
 Now, we consider two linear Laplacian
operators on ${\cal F}_{_{\pi}}(G)$ as follows
$$\dirL \isdef id-K \quad {\rm and} \quad \siml \isdef id-\overline{K},$$
where $id$ is the identity operator. It is clear that $\overline{K}$
and $\siml$ are self-adjoint operators on ${\cal F}_{_{\pi}}(G)$ by
definition, while $K$ and $\dirL$ may not be necessarily so, unless
$K=K^*$ and $\dirL=\siml$. Hence, when $|V(G)|=n$, one may order
all real eigenvalues of $\siml$ as
\begin{equation}\label{LAMBDASEQ}
0=\lambda_{_{1}} \leq \lambda_{_{2}} \leq \cdots \leq
\lambda_{_{n}}.
\end{equation}
(At times we may use superscripts as
$\lambda^{^{G}}_{_{2}}$ or $\lambda^{^{K}}_{_{2}}$ to refer to the
graph or the kernel when details are
clear from the context.)\\
Also, it is a well-known fact (Perron-Frobenius theorem) that for a
strongly connected graph $G$, the eigenspace corresponding to the
eigenvalue $0=\lambda^{^G}_{_{1}}$ is one-dimensional and is
generated by the constant vector ${\bf 1}$. Moreover, for any
$n \times n$ self-adjoint matrix $A$, and for any $1 \leq k \leq n$, by Courant--Fischer variational principle
(see \cite{SAL97}), one may write
\begin{equation}
\lambda^{^A}_{_{k}}=\displaystyle{\min_{_{W \in {\cal W}_{_{k}}}}}
\displaystyle{\max_{_{\ 0 \not = f \in W}}} \left \{\frac{ \langle Af,f\rangle }{
\|f\|^{^2}} \right \}= \displaystyle{\max_{_{\ W \in {\cal
W}^{\bot}_{_{k-1}}}}} \displaystyle{\min_{_{\ 0 \not = f \in W}}}
\left \{\frac{\langle Af,f\rangle}{\|f\|^{^2}} \right \},
\end{equation}
in which
$${\cal W}_{_{k}} \isdef
\{W  \ \ | \ \ {\rm dim}(W) \geq k \}, \quad {\cal
W}^{\bot}_{_{k}} \isdef \{W  \ \ | \ \ {\rm dim}(W^{\bot}) \leq k
\},$$ and $\langle f,g\rangle$ is the inner product of the space on which $A$ is defined and is self-adjoint.
\subsection{Gradients, energy and their properties}
Given a graph $G$, one may define the {\it directed}, {\it
classical}, and {\it symmetric} gradients, respectively, as follows
\begin{itemize}
\item{$\dirg : {\cal F}_{_{\pi}}(G) \longrightarrow {\cal F}_{_{\phi}}(G)$ as $\dirg f (uv) \isdef (f(u)-f(v))^+$.}
\item{$\simg : {\cal F}_{_{\pi}}(G) \longrightarrow {\cal F}_{_{\phi}}(G)$ as $\simg f (uv) \isdef f(u)-f(v)$.}
\item{$\ssimg : {\cal F}_{_{\pi}}(\overline{G}) \longrightarrow {\cal F}_{_{\overline{\phi}}}(\overline{G})$ as $\ssimg f (uv) \isdef |f(u)-f(v)|$.}
\end{itemize}
In the rest of this paper we will adopt the following
framework. \\ \ \\
{\bf Assumption:}
Following our previous notational assumption, hereafter, given any directed graph $G$  (possibly with loops),
we will be working with the measure spaces $(V(G),\pi)$ and
$(E(G),\phi)$  as well as $(V(\overline{G}),\pi)$ and
$(E(\overline{G}),\overline{\phi})$ for the corresponding
undirected graph, (note that the last case also covers the case
of simple graphs). In our notations, the subscript determines the
function space under consideration (e.g. ${\cal F}_{_{\phi}}(G)$
stands for the set of all real functions defined on $E(G)$, the
set of edges of a given graph $G$, equipped with an inner product
weighted by $\phi$). Thus we will be working within the frameworks
$[G,(V(G),\pi),(E(G),\phi),\dirg,\simg,\dirL]$  and
$[\overline{G},(V(\overline{G}),\pi),(E(\overline{G}),\overline{\phi}),\ssimg,\siml]$
for the (in general directed) graph $G$, and the corresponding undirected graph, respectively. \\  \ \\
 It ought to be noted that considering the inner-product space
equipped with the weighted inner-product $\langle.,.\rangle_{_{\pi}}$ has the
advantage of reflecting parts of the global structural properties
of the base graph in the spectrum of the corresponding Laplacian
operator (e.g. see \cite{ALGCIRC,NODAL,JGT,FIL91}), while this is not necessarily true when one uses the
ordinary inner-product of ${\mathbb R}^n$ or symmetrization by the square root of the degree matrix (e.g. as in \cite{CH97,CH05,CGY00}).
Let us start with the following well-known result.
\begin{lem}\label{LAPLACADJ}
For a given graph $G$, the classical gradient, $\simg$, is a linear
operator and has an adjoint $\simg^* : {\cal F}_{_{\phi}}(G)
\longrightarrow {\cal F}_{_{\pi}}(G) $ defined as
$$\simg^* f (u)  \isdef \frac{1}{\pi(u)} \left ( \displaystyle{\sum_{_{uv\in E(G)}}} f(uv) \phi(u,v)
    - \displaystyle{\sum_{_{vu\in E(G)}}}  f(vu) \phi(v,u)  \right ).$$
Moreover, $2\siml= \simg^* \simg$.
\end{lem}
\begin{proof}{
Verification of the adjunction is straightforward. For the second equality, we have,
$$(\siml f)(u)=\frac{1}{\pi(u)} \displaystyle{\sum_{_{v \in
V(G)}}} (f(u)-f(v))\ \overline{\phi}(u,v)=\frac{1}{\pi(u)}
\displaystyle{\sum_{_{v \in V(G)}}} \simg f(uv)\
\overline{\phi}(u,v)=\frac{1}{2}\ \simg^* \simg f(u).$$ Also, note
that,
$$\langle 2\siml f,g\rangle_{_{\pi}}=\langle \simg^* \simg f,g\rangle_{_{\pi}}=\langle \simg f,\simg g\rangle_{_{\phi}},$$
holds for all $f,g \in {\cal F}_{_{\pi}}(G)$.
 }\end{proof}
The simple but important statement of Lemma~\ref{LAPLACADJ} in a way
presents the symmetrization process of constructing the undirected
symmetric graph $\overline{G}$ from a given graph $G$, in an
analytic sort of way. In other words, starting from a kernel $K$ on
a base graph $G$, and considering the operators $\simg$ and
$\simg^*$, one may construct the symmetric Laplacian operator as
$\siml=\frac{1}{2}\simg^*\simg$ that introduces a new kernel whose base graph
is $\overline{G}$. Also, a classical and interesting fact is that if
one starts from a graph $G$ and considers the conservation of energy
as Kirchhoff's node and loop laws, then one finds a Poisson's
equation relating the current and voltage (i.e. potential) functions
whose basic operator is the symmetric Laplacian $\siml$ on
$\overline{G}$ (e.g. see \cite{SOA94}). Hence, in this
sense, conservation of energy naturally is linked to connectivity through the symmetric model.\\
With this background, one of our main objectives can be described as
finding methods that can reflect some of the {\it connectivity} properties
of $G$ in its related symmetric model $\overline{G}$, through the
self-adjoint or symmetric operators defined on it. Therefore, it is
natural to concentrate on well-behaved or induced operators on
$\overline{G}$ (e.g. $\ssimg : {\cal F}_{_{\pi}}(\overline{G})
\longrightarrow {\cal F}_{_{\overline{\phi}}}(\overline{G})$) and
consider their relationships to those of $G$. The following lemma
summarizes some of the basic properties of these operators for
further reference. Specially, note that Equation~(\ref{LEMMA2A}) is a consequence of part $(a)$.
\begin{lem}\label{LEMGER}
For any given graph $G$ with a kernel $K$ on it, and $f \in {\cal F}_{_{\pi}}(G)$,
\begin{itemize}
\item[a {\rm )}]{$\| \dirg f\|_{_{1,\phi}} = \frac{1}{2}\ \|\simg f\|_{_{1,\phi}}=\| \ssimg f \|_{_{1,\overline{\phi}}}.$}
\item[b {\rm )}]{$\|\dirg f\|_{_{1,\phi}}=\|\dirg f^+\|_{_{1,\phi}}+\|\dirg f^-\|_{_{1,\phi}}.$}
\item[c {\rm )}]{$\frac{1}{2}\ \|\simg f\|^{^2}_{_{2,\phi}}= \|\ssimg f\|^{^2}_{_{2,\overline{\phi}}}=\langle\siml f,f\rangle_{_{\pi}}
                  =\langle\dirL f,f\rangle_{_{\pi}}.$}
\end{itemize}
\end{lem}
\begin{proof}{For $(a)$ note that $2\dirg f(uv)=\big(f(u)-f(v)\big)+|f(u)-f(v)|$. Since  $\phi$ is a flow on $E(G)$,
we have $$\displaystyle{\sum_{_{uv\in E(G)}}} \big(f(u)-f(v)\big)\ \phi(u,v)=\sum_{u\in V(G)}f(u)\ \big(\dirp(\{u\})-\rdirp(\{u\})\big)=0$$ and
consequently,
\begin{eqnarray*}
\|\dirg f\|_{_{1,\phi}} &=&\displaystyle{\sum_{_{uv \in E(G)}}}
\dirg f(uv)\ \phi(u,v)=\frac{1}{2} \displaystyle{\sum_{_{uv \in
E(G)}}} |f(u)-f(v)|\ \phi(u,v)\\ &=&\frac{1}{2}\sum_{_{uv\in
E(\overline{G})}} \ssimg f(uv)\ \big(\phi(u,v)+\phi(v,u)\big)
=\|\ssimg f\|_{_{1,\overline{\phi}}}.
\end{eqnarray*}
Equality in $(b)$ is clear. Also, $(c)$ follows from
Lemma~\ref{LAPLACADJ} and the following equalities,
\begin{eqnarray*}
\langle(id-K) f,f\rangle_{_{\pi}}&=&\langle(id-\frac{1}{2}(K+K^*)) f,f\rangle_{_{\pi}}=\frac{1}{2}\|\simg f\|^{^2}_{_{2,{\phi}}}\\
& = & \frac{1}{2}\displaystyle{\sum_{_{uv
\in E(\overline{G})}}} |f(u)-f(v)|^2\ \big(\phi(u,v)+\phi(v,u)\big)=
\|\ssimg f\|^{^2}_{_{2,{\overline{\phi}}}}.
\end{eqnarray*}
 }\end{proof}
Clearly, in this approach, one needs some relations between the
energy (Dirichlet) forms and different norms of the operators to
construct the necessary connections needed. The following two
lemmas demonstrate the most basic relationships.
\begin{lem}
For every $f \in {\cal F}_{_{\pi}}(G)$ we have
\begin{itemize}
\item[a {\rm )}]{$ \|\simg f\|_{_{1,\phi}} \leq \| \simg f \|_{_{2,\phi}}.$ }
\item[b {\rm )}]{$\| \dirg f\|_{_{1,\phi}}=\| \ssimg f\|_{_{1,\overline{\phi}}} \leq
\frac{\sqrt{2}}{2}\ \| \ssimg f \|_{_{2,\overline{\phi}}}.$ }
\end{itemize}
\end{lem}

\begin{proof}{Since $\phi$ is a flow we have $\displaystyle{\sum_{_{uv \in E(G)}}} \ \phi(u,v)=1$,
and consequently, by Cauchy-Schwarz inequality,
$$
\begin{array}{rl}
\displaystyle{\sum_{_{uv \in E(G)}}} |f(u)-f(v)|^2 \ \phi(u,v) &
=\left ( \displaystyle{\sum_{_{uv \in E(G)}}} |f(u)-f(v)|^2 \
\phi(u,v) \right )
\left (  \displaystyle{\sum_{_{uv \in E(G)}}} \ \phi(u,v) \right )\\
&\\
&\geq \left ( \displaystyle{\sum_{_{uv \in E(G)}}} |f(u)-f(v)| \
\phi(u,v) \right )^2.
\end{array}
$$
Part $(b)$ follows by a similar discussion.
 }\end{proof}

\begin{lem}\label{LEMNORM2}
For every $0 \not = f \in {\cal F}_{_{\pi}}(G)$ we have
\begin{itemize}
\item[a {\rm )}]{$\frac{\|\nabla f^2\|_{_{1,\phi}}}{\|f^2\|_{_{1,\pi}}} \leq 2\ \frac{\| \nabla f
\|_{_{2,\phi}}}{\| f \|_{_{2,\pi}}}.$ }
\item[b {\rm )}]{$\frac{\|\stackrel{\rightarrow}{\nabla} f^2\|_{_{1,\phi}}}{\|f^2\|_{_{1,\pi}}}=\frac{\|\overline{\nabla}
f^2\|_{_{1,\overline{\phi}}}}{\|f^2\|_{_{1,\pi}}} \leq \sqrt{2}\
\frac{\| \overline{\nabla} f \|_{_{2,\overline{\phi}}}}{\| f
\|_{_{2,\pi}}}.$ }
\end{itemize}
\end{lem}
\begin{proof}{The proof is clear by Lemma~\ref{LEMGER}(c), Cauchy-Schwarz inequality \\
(e.g. $(a+b)^2 \leq 2(a^2+b^2)$) and the following
$$
\begin{array}{rl}
\left ( \frac{\sqrt{2}\ \| \overline{\nabla} f
\|_{_{2,\overline{\phi}}}}{\| f \|_{_{2,\pi}}} \right )^2 & =
\frac{\displaystyle{\sum_{_{uv \in E(G)}}} |f(u)-f(v)|^2\
\phi(u,v)}{\displaystyle{\sum_{_{u \in V(G)}}} |f(u)|^2\
{\pi}(u)}\times \frac{\displaystyle{\sum_{_{uv \in E(G)}}}
|f(u)+f(v)|^2\ \phi(u,v)}{\displaystyle{\sum_{_{uv \in E(G)}}}
|f(u)+f(v)|^2\ \phi(u,v)}\\
 &\\
 &\geq \frac{\left ( \displaystyle{\sum_{_{uv \in E(G)}}}
|f(u)^2-f(v)^2|\ \phi(u,v)\right )^2}{\displaystyle{4}\ \left (
\displaystyle{\sum_{_{u \in V(G)}}} |f(u)|^2\ {\pi}(u) \right )
^2} = \left (\frac{\|\stackrel{\rightarrow}{\nabla}
f^2\|_{_{1,\phi}}}{\|f^2\|_{_{1,\pi}}} \right )^2.
\end{array}
$$
 }\end{proof}
\section{The isoperimetric spectrum}\label{ISOSPEC}
In this section we concentrate on the mean isoperimetric constant and its
generalization. In this regard, our point of view is to consider a
generalization that is, firstly, well-behaved computationally, and
secondly, can present a good relation to the classical
eigenvalues.
Throughout the section, $K$ is the kernel of a fixed Markov chain on
the base graph $G$ as before, and $\pi$ is a nowherezero stationary
distribution for this kernel.\\
 It is a well-known fact from random-matrix
theory and the recent literature that the behavior of the
classical spectrum of the Laplacian operator is quite hard to
predict and, as a matter of fact, is related to some deep
problems in contemporary mathematics \cite{DIA03}. In our
opinion, one possible approach in this direction is to analyze a
smooth function of the spectrum, that in a way contains a fair
amount of data, rather than the eigenvalues themselves.
Naturally, the most simple candidate for such a function can be
considered to be the {\it arithmetic mean}, and consequently,
there seems to be a fair chance that the behavior of the
mean-spectrum, whose $n$th element is the mean of the first
$n$ eigenvalues, be more well-behaved than the spectrum itself. We
should also mention the results of J.~B.~Hiriart-Urruty and D.~Ye
\cite{HY95} that, in a sense, justifies this approach.\\
Therefore, based on the above-mentioned approach we will focus on
the mean version of the isoperimetric constant and will generalize
it as the most natural $L^{1}$ counterpart of the mean eigenvalue.
It is interesting to note that this generalization leads to a
definition for the $n$th isoperimetric number which is based
on taking a minimum over all $n$-{\it disjoint} subsets of the
ground-space, rather than its $n$-partitions, and also satisfies a
Federer-Fleming-type theorem (Theorem~\ref{IOTAEQGAMMA}). This
difference, although disguised in the classical case $k=2$ (see Proposition~\ref{INEQCOR}), seems to
be quite nontrivial in general and will be our main motivation for
the definition of a {\it supergeometric graph}.\\
It is not hard to check that there is a straightforward translation
of almost all results of this section to the case of max-isoperimetric constants (see Section~\ref{SECGENCHEEGER}
for a precise definition) or the case of
compact Riemannian manifolds (considering appropriate modifications).\\
In what follows we introduce the generalized isoperimetric number
(in the mean case), and we investigate some of its basic properties.
To begin, we set a couple of notations.
The set ${\cal D}_{_{n}}(G)$ is
defined to be the set of all $n$-sets $\{Q_{_{1}},\ldots,Q_{_{n}}\}$
with $\emptyset \not = Q_{_{i}} \subseteq V(G)$ for all $1 \leq i
\leq n$ such that for every pair $1 \leq i < j \leq n$ we have
$Q_{_{i}} \cap Q_{_{j}} = \emptyset$. The set of {\it
$n$-partitions} of a graph $G$, which is denoted by ${\cal
P}_{_{n}}(G)$, is the subset of ${\cal D}_{_{n}}(G)$ that contains
all $n$-sets $\{Q_{_{1}},\ldots,Q_{_{n}}\}$ for which
$\cup_{_{i=1}}^{^{n}} Q_{_{i}}=V(G)$.\\
Now, we define the generalized mean isoperimetric constants as,
\begin{defin}{Given a graph $G$ and a kernel $K$,
the $n$th (mean) {\it isoperimetric constant} of $G$ (with respect
to $K$) is defined as follows
$$\iota_{_{n}}(G,K) \isdef
\displaystyle{\min_{_{ \{Q_{_{i}}\}^{^{n}}_{_{1}} \in  {\cal
D}_{_{n}}(G) }} } \ \frac{1}{n} \left (
\displaystyle{\sum^{n}_{i=1}} \frac{\dirp
(Q_{_{i}})}{\pi(Q_{_{i}})} \right ).$$
 Also, considering the partitions, we define the following related constant,
$$\tilde{\iota}_{_{n}}(G,K) \isdef \displaystyle{\min_{_{
\{Q_{_{i}}\}^{^{n}}_{_{1}} \in  {\cal P}_{_{n}}(G) }} } \
\frac{1}{n} \left ( \displaystyle{\sum^{n}_{i=1}} \frac{ \dirp
(Q_{_{i}})}{\pi(Q_{_{i}})} \right ).$$ We may exclude the kernel
$K$ from the notations, when it is fixed or is clear from the
\linebreak
context.
 }\end{defin}
 Some basic properties of the mean isoperimetric spectrum are
 stated in the following proposition.
\begin{pro}\label{INEQCOR}
For any graph $G$ $($and a given kernel $K$ on it$)$ and for all \\
$1 \leq n \leq |V(G)|$ we have,
\begin{itemize}
\item[a{\rm )}]{$\iota_{_n}(G) \leq \tilde{\iota}_{_{n}}(G)\leq (1-\frac{1}{n})
\  \iota_{_{n}}(G) +\frac{1}{n}.$ }
\item[b{\rm )}]{$\tilde{\iota}_{_{2}}(G)=\iota_{_{2}}(G).$ }
\item[c{\rm )}]{$\tilde{\iota}_{_{n}}(G) \leq (1-\frac{1}{n^2})\ \tilde{\iota}_{_{n+1}}(G) < \tilde{\iota}_{_{n+1}}(G).$ }
\item[d{\rm )}]{$\iota_{_{n}}(G) \leq \iota_{_{n+1}}(G).$ }
\end{itemize}
\end{pro}

\begin{proof}{The left-hand inequality of part $(a)$ is clear by
definitions. Assume that $\iota_{_n}(G)$ is achieved by choosing $\{Q_{_i}\}_{_1}^{^n} \in {\cal D}_{_{n}}(G)$ and suppose that $\displaystyle{\frac{\dirp (Q_{_n})}{\pi(Q_{_n})}}$ is maximum of all $\displaystyle{\frac{\dirp (Q_{_i})}{\pi(Q_{_i})}}$,  $i\in{\cal I}_n$.
Then the partition
$\{Q'_{_i}\}_{_1}^{^n} \in {\cal P}_{_{n}}(G)$ with $Q'_{_i} \isdef Q_{_i}$ for all $i \in {\cal
I}_{_{n-1}}$ and\\
 $Q'_{_n} \isdef V(G) - (\cup_{_{i=1}}^{^{n-1}}
Q_{_i})$ will satisfy
\[\tilde{\iota}_{_n}(G)\leq \frac{1}{n} \left(\frac{\dirp (Q'_{_n})}{\pi(Q'_{_n})}+\sum_{i=1}^{n-1}\frac{\dirp (Q_{_i})}{\pi(Q_{_i})}
\right)\leq \frac{1}{n}(1+(n-1)\ \iota_{_n}(G))=(1-\frac{1}{n})\  \iota_{_{n}}(G)+\frac{1}{n}.\]
For part~(b), let
$\{Q_{_1}, Q_{_2}\} \in {\cal D}_{_{2}}(G)$ be such that $\iota_{_2}(G)$
is achieved and let\\
 $Q^{^*} \isdef V(G) - (Q_{_1}\cup Q_{_2})$.
Without loss of generality,  assume that $\dirp(Q^{^*},Q_{_1})\leq
\dirp(Q_{_2},Q^{^*})$. Then for the partition $\{Q'_{_1},Q'_{_2}\}$
with $Q'_{_1} \isdef Q_{_1}$ and $Q'_{_2} \isdef Q_{_2}\cup Q^{^*}$ we have
\[2\ \tilde{\iota}_{_2}(G)\leq \sum_{i=1}^2\frac{\dirp (Q'_{_i})}{\pi(Q'_{_i})}
= \frac{\dirp (Q_{_1})}{\pi(Q_{_1})}+\frac{\dirp (Q_{_2},Q_{_1})+\dirp
(Q^{^*},Q_{_1})}{\pi(Q_{_2})+\pi(Q^*)}\leq\sum_{i=1}^2\frac{\dirp
(Q_{_i})}{\pi(Q_{_i})}=2\ \iota_{_2}(G).\] The reverse
inequality is clear from the definitions. \\
For part (c), let
$\{Q_{_i}\}_{_1}^{^{n+1}} \in {\cal P}_{_{n+1}}(G)$ be a partition such that
$\tilde{\iota}_{_{n+1}}(G)$ is achieved. For every pair of
indices $\{j,k\} \subseteq {\cal I}_{_{n+1}}$ define
$$T^{^n}_{_{j,k}} \isdef \frac{1}{n} \left ( \frac{\dirp (Q_{_{j}} \cup Q_{_{k}})}{\pi(Q_{_{j}} \cup Q_{_{k}})}
 + \displaystyle{\sum_{i \in {\cal I}_{_{n+1}}-\{j,k\}}} \frac{\dirp (Q_{_{i}})}{\pi(Q_{_{i}})}  \right ),$$
and let $w_{_{j,k}}=\pi(Q_{_j}\cup Q_{_k})$. Note that $\sum_{1 \leq j < k
\leq n+1}\ w_{_{j,k}}=n.$ Also
$$
\begin{array}{rl}
n\ \displaystyle{\sum_{1 \leq j < k \leq n+1}}\
w_{_{j,k}}T^{^n}_{_{j,k}} &=
                                             \displaystyle{\sum_{1\leq j < k \leq n+1}}\
                                                \dirp (Q_{_{j}} \cup
                                                Q_{_{k}})\\
                                                &\quad +
                                                \displaystyle{\sum_{1 \leq j < k \leq n+1}}\
                                                 \displaystyle{\sum_{i \in {\cal I}_{_{n+1}}-\{j,k\}}}\ \pi(Q_{_{j}} \cup Q_{_{k}})
                                                 \frac{\dirp (Q_{_{i}})}{\pi(Q_{_{i}})}\\
                                                & = (n-1) \displaystyle{\sum_{i=1}^{n+1}}\  \dirp
                                                (Q_{_{i}})\\
                                                &\quad +
                                                     \displaystyle{\sum_{i=1}^{n+1}}\
                                                     (n-1)(1-\pi(Q_{_{i}}))
                                                     \frac{\dirp (Q_{_{i}})}{\pi(Q_{_{i}})}\\
                                                &=
                                                      (n-1) \displaystyle{\sum_{i=1}^{n+1}}\
                                                     \frac{\dirp (Q_{_{i}})}{\pi(Q_{_{i}})}.
\end{array}
$$
Thus,
\[
 \tilde{\iota}_{_n}(G) \leq
\displaystyle{\min_{\{j,k\} \subseteq  {\cal I}_{_{n+1}}}}\
T^{^n}_{_{j,k}} \leq \frac{1}{n} \displaystyle{\sum_{1 \leq j < k
\leq n+1}}\ w_{_{j,k}}T^{^n}_{_{j,k}} \leq \frac{n^2-1}{n^2} \
\tilde{\iota}_{_{n+1}}(G).
\]
For part $(d)$, let
 $\{Q_{_{i}}\}_{_{i \in {\cal I}_{_{n+1}}}} \in{\cal D}_{_{n+1}}(G)$
 be chosen such that $\iota_{_{n+1}}(G)$ is achieved, and moreover, without loss of generality, assume that
 $$\frac{\dirp (Q_{_{1}})}{\pi(Q_{_{1}})} \leq \cdots
\leq \frac{\dirp (Q_{_{n+1}})}{\pi(Q_{_{n+1}})}.$$
 Then clearly,
 $$ \iota_{_{n}}(G) \le \frac{1}{n} \sum_{i=1}^{n} \frac{\dirp (Q_{_{i}})}{\pi(Q_{_{i}})} \leq
    \frac{1}{n+1} \sum_{i=1}^{n+1} \frac{\dirp (Q_{_{i}})}{\pi(Q_{_{i}})} = \iota_{_{n+1}}(G).$$
 }\end{proof}
\section{\label{EXAMPLES}Some examples and special cases}
As the first example, let us consider the case  $n=2$.
\begin{exm}{ Let $G$ be a given directed graph. We have
$$
\begin{array}{rl}
\tilde{\iota}_{_{2}}(G) &= \displaystyle{\min_{_{
                          \{Q_{_{i}}\}^{^{2}}_{_{1}} \in  {\cal P}_{_{2}}(G) }} } \
                          \frac{1}{2} \left ( \frac{ \dirp (Q_{_{1}})}{\pi(Q_{_{1}})}  +  \frac{ \dirp (Q_{_{2}})}{\pi(Q_{_{2}})} \right )\\
                        &=\displaystyle{\min_{_{ Q \subseteq V(G)}} } \
                           \frac{1}{2} \left ( \frac{ \dirp (Q)}{\pi(Q)}  +  \frac{ \dirp (Q^c)}{\pi(Q^c)} \right )\\
                        &=\displaystyle{\min_{_{ Q \subseteq V(G)}} } \
                           \frac{1}{2} \left ( \frac{ \dirp (Q)}{\pi(Q)}  +  \frac{ \dirp (Q)}{(1-\pi(Q))} \right )\\
                        &=\displaystyle{\min_{_{ Q \subseteq V(G)}} } \
                            \frac{ \dirp (Q)}{2 \pi(Q)(1-\pi(Q))},\\
\end{array}
$$
which is the (mean version) of the classical Cheeger constant. Therefore,
since we have $\tilde{\iota}_{_{2}}(G)=\iota_{_{2}}(G)$ by Proposition~\ref{INEQCOR}$(b)$, our definition of
the isoperimetric number for the classical case is justified.
 }\end{exm}
Given a strongly connected  directed graph $G$, we define the natural
random walk on $G$ by
\[K(u,v)=\left\{\begin{array}{ll}
\frac{1}{d^+(u)} & uv\in E(G)\\
0 & uv\not\in E(G),
\end{array}\right.\]
where $d^+(u)$ stands for the out-degree of vertex $u$. If the
graph $G$ is Eulerian, i.e. for every vertex $u\in V(G)$, we have
$d^+(u)=d^-(u)$, then one can easily see that the distribution
$\pi$, defined by $\pi(u)=\frac{d^+(u)}{|E(G)|}$ is the unique
stationary distribution for the natural random walk on $G$ which
induces the flow $\phi$ on $G$, defined by $\phi(u,v) \isdef \frac{1}{|E(G)|}$,
whenever $uv\in E(G)$ and zero elsewhere. Hence, for every
subset $Q\subseteq V(G)$, we have
\begin{equation}\label{NATURAL}
\frac{\dirp (Q)}{\pi(Q)}= \frac{|\odirb(Q)|}{d^+(Q)},
\end{equation}
where $d^+(Q) \isdef \sum_{u\in Q} d^+(u)$. Note that for any connected
undirected graph $G$, the symmetric directed graph $\sdg[G]$ is Eulerian, which shows that these arguments are valid in the case of undirected connected graphs as well.
\subsection{Geometric graphs}\label{GEOGRAPH}
By Proposition~\ref{INEQCOR}, for any given directed graph $G$
with a kernel $K$ and a  nowherezero stationary
distribution $\pi$ on it, one can talk about the isoperimetric
spectrum,
$$0=\iota_{_{1}}(G,K) \leq \iota_{_{2}}(G,K) \leq \ldots \leq \iota_{_{|V(G)|}}(G,K) \leq 1.$$
Also, note that if $G$ has no loops then $\iota_{_{|V(G)|}}(G,K) = 1$.
On the other hand, by definitions,
for any given graph $G$ and for all $1 \leq n \leq |V(G)|$, we have
$\iota_{_{n}}(G,K) \leq \tilde{\iota}_{_{n}}(G,K)$, that motivates the following
definition (also see Theorem~\ref{IOTAEQGAMMA} and Section~\ref{KMEANS}).
 \begin{defin}{
A graph $G$ is said to be {\it $n$-geometric} with respect to a
kernel $K$, if
$$\iota_{_{n}}(G,K) = \tilde{\iota}_{_{n}}(G,K).$$
 A graph $G$ is said to be {\it supergeometric} with respect to a
kernel $K$, if it is $n$-geometric with respect to $K$, for every $2
\leq n \leq |V(G)|$.
 }\end{defin}
 By definition
and  Proposition~\ref{INEQCOR}($b$), any strongly connected graph is $2$-geometric
(with respect to any given kernel $K$).
An easy observation is that for any graph without loops and with
respect to any kernel,
$$|Q|=1 \quad \Rightarrow \quad \frac{ \dirp (Q)}{\pi(Q)}=1.$$
This, for instance, shows that all simple graphs on a set of $6$
vertices are supergeometric. In what follows we elaborate on going
through the details of computing the mean isoperimetric constants of some
well-known graphs, to provide examples of supergeometric graphs as
well as cases for which the graph is far from being geometric.
\begin{exm}{\label{COMPLETE}
In this example we compute the  isoperimetric spectra of
complete graphs and complete bipartite graphs
with respect to their natural
random walks,  and we show that they are supergeometric.\\
 By Equation~(\ref{NATURAL}), for any
$\{Q_{_{i}}\}^{^{n}}_{_{1}} \in {\cal D}_{_{n}}(V({\sf K}_{_{t}})),$
with $|Q_{_{i}}|=t_{_{i}}$, we have
$$\sum_{i=1}^n \frac{\dirp (Q_{_i})}{ \pi(Q_{_i})}=\sum_{i=1}^n\frac{t_{_i}(t-t_{_i})}{ (t-1)t_{_i}}=\frac{tn-\sum_{_{i}} t_{_{i}}}{(t-1)},$$
which is clearly minimized when $\{Q_{_{i}}\}^{^{n}}_{_{1}} \in
{\cal P}_{_{n}}({\sf K}_{_{t}}).$ Therefore for all $n\in{\cal I}_{_t}$ we have
$$\iota_{_{n}}({\sf K}_{_{t}})=\tilde{\iota}_{_{n}}({\sf K}_{_{t}})=\frac{t(n-1)}{n(t-1)},$$
and complete graphs are supergeometric.\\
Now, let $X$ and $Y$ be the two parts of the graph
${\sf K}_{_{r,s}}$, with $|X|=r, |Y|=s$, and let
$\{Q_{_{i}}\}^{^{n}}_{_{1}} \in {\cal D}_{_{n}}(V({\sf
K}_{_{r,s}})),$ be such that $|Q_{_{i}}\cap X|=x_{_{i}}$ and
$|Q_{_{i}}\cap Y|=y_{_{i}}$. By Equation~(\ref{NATURAL}), we have
\[\frac{1}{n}\sum_{i=1}^n \frac{\dirp (Q_{_i})}{\pi(Q_{_i})}= \frac{1}{n}\sum_{i=1}^n\frac{x_{_i}(s-y_{_i})+y_{_i}(r-x_{_i})}{sx_{_i}+ry_{_i}}=
1-\frac{2}{n}\sum_{i=1}^n\frac{x_{_i}y_{_i}}{sx_{_i}+ry_{_i}}.\]
First, note that the function
$\frac{x_{_i}y_{_i}}{sx_{_i}+ry_{_i}}$ is increasing with respect
to both $x_{_i}$ and $y_{_i}$, and consequently, one deduces that complete bipartite graphs are supergeometric.\\
Furthermore, as a special case, let $s$ be a multiple of
$r$, where we want to maximize the function
$\frac{x_{_i}y_{_i}}{sx_{_i}+ry_{_i}}$ under constraints
$\sum_{i=1}^n x_{_i}=r$ and $\sum_{i=1}^n y_{_i}=s$. Using Lagrange method we can see that the function is maximized when $sx_{_i}=ry_{_i}$,
for every $i\in {\cal I}_{_n}$. Thus, for every $n\in {\cal I}_{_r}$, we
have
\[\iota_{_n}({\sf K}_{_{r,s}})=\tilde{\iota}_{_n}({\sf K}_{_{r,s}})=1-\frac{2}{n}\sum_{i=1}^n \frac{sx_{_i}}{r(s+s)}=1-\frac{1}{n}.\]
 }\end{exm}
\begin{exm}{
Let $G_{_t}=(V,E)$ be the directed cycle with loops, where
$V \isdef {\mathbb Z}/t{\mathbb Z}$ and $E \isdef \{(i,i), (i,i+1)|\ i\in V\}$.
Considering the natural random walk on $G_{_t}$, for every $Q\subset V$ we have
$$\frac{\dirp (Q)}{\pi(Q)}= \frac{s}{2|Q|},$$
where $s \isdef |\{i\in Q|\ i+1\not\in Q\}|$.
So when $Q\subsetneqq V$ is a nonvoid set of
consecutive numbers, this quotient is minimized and is equal to
$\frac{1}{2|Q|}$. Thus for every $2\leq n\leq t$,
\[\iota_{_n}(G_{_t})=\displaystyle{\min_{\{Q_{_{i}}\}^{^{n}}_{_{1}} \in {\cal D}_{_{n}}(G_{_{t}})}}\
\frac{1}{n}\sum_{i=1}^n \frac{\dirp (Q_{_i})}{\pi(Q_{_i})}=
\displaystyle{\min_{\{Q_{_{i}}\}^{^{n}}_{_{1}} \in {\cal D}_{_{n}}(G_{_{t}})}}\  \frac{1}{2n}\sum_{i=1}^n \frac{1}{|Q_{_i}|},\]
which is clearly minimized when $\{Q_{_{i}}\}^{^{n}}_{_{1}} \in
{\cal P}_{_{n}}(G_{_n})$. Consequently, the graph $G_{_t}$ is supergeometric and if $t=\lfloor\frac{t}{n}\rfloor
n+r$, for some $r$, then for every $2\leq n\leq t$,
$$\iota_{_{n}}({G}_{_{t}})=\tilde{\iota}_{_{n}}({G}_{_{t}})=\frac{1}{2n}\left( \frac{n-r}{\lfloor{\frac{t}{n}}\rfloor}+\frac{r}{\lfloor{\frac{t}{n}}\rfloor+1}\right).$$
 }\end{exm}
In the following example we introduce a graph which is $2$-geometric but not $3$-geometric.
\begin{figure}[t]
\unitlength=2pt
\begin{center}
\begin{picture}(90,52)
\drawshadedcircle{8.0}{10.0}{4.0}{}{1.0}
\drawshadedcircle{20.0}{10.0}{4.0}{}{1.0}
\drawshadedcircle{32.0}{10.0}{4.0}{}{1.0}
\drawshadedcircle{46.0}{10.0}{0.0}{}{1.0}
\drawshadedcircle{44.0}{10.0}{4.0}{}{1.0}
\drawshadedcircle{56.0}{10.0}{4.0}{}{1.0}
\drawshadedcircle{68.0}{10.0}{4.0}{}{1.0}
\drawshadedcircle{80.0}{10.0}{4.0}{}{1.0}
\drawshadedcircle{44.0}{22.0}{4.0}{}{1.0}
\drawshadedcircle{44.0}{34.0}{4.0}{}{1.0}
\drawshadedcircle{44.0}{46.0}{4.0}{}{1.0}
\drawpath{8.0}{10.0}{20.0}{10.0} \drawpath{20.0}{10.0}{80.0}{10.0}
\drawpath{44.0}{46.0}{44.0}{10.0}
\drawcenteredtext{8.0}{14.0}{$v_{_1}$}
\drawcenteredtext{20.0}{14.0}{$v_{_2}$}
\drawcenteredtext{32.0}{14.0}{$v_{_3}$}
\drawcenteredtext{80.0}{14.0}{$v_{_4}$}
\drawcenteredtext{68.0}{14.0}{$v_{_5}$}
\drawcenteredtext{56.0}{14.0}{$v_{_6}$}
\drawcenteredtext{38.0}{46.0}{$v_{_7}$}
\drawcenteredtext{38.0}{34.0}{$v_{_8}$}
\drawcenteredtext{38.0}{22.0}{$v_{_9}$}
\drawcenteredtext{44.0}{6.0}{$v_{_{10}}$}
\end{picture}
\end{center}
\caption{See Example~\ref{NONGEO}. \protect\label{STARG}}
\end{figure}
\begin{exm}{\label{NONGEO}
Consider the simple graph $G$ of Figure~\ref{STARG} along  with its natural random
walk, where we are going to compute $\iota_{_3}(G)$ and
$\tilde{\iota}_{_3}(G)$. By considering disjoint sets
$$\{A_{_i} \isdef \{v_{_{3i-2}},v_{_{3i-1}},v_{_{3i}}\}|\ i=1,2,3\},$$
 and
the partition
$$\{B_{_1} \isdef \{v_{_1},v_{_2},v_{_3}\},
B_{_2} \isdef \{v_{_4},v_{_5},v_{_6}\},
B_{_3} \isdef \{v_{_7},v_{_8},v_{_9},v_{_{10}}\}\},$$
 we have
\begin{eqnarray*}
\iota_{_3}(G)&\leq& \frac{1}{3}\sum_{i=1}^3 \frac{\dirp(A_{_i})}{
\pi(A_{_i})}=\frac{1}{3}(\frac{1}{5}+\frac{1}{5}+\frac{1}{5})=\frac{1}{5},\\
\tilde{\iota}_{_3}(G)&\leq& \frac{1}{3}\sum_{i=1}^3 \frac{\dirp(B_{_i})}{ \pi(B_{_i})}=\frac{1}{3}(\frac{1}{5}+\frac{1}{5}+\frac{2}{8})=\frac{13}{60}.
\end{eqnarray*}
It is easy to verify the following claims (by a case study) for a subset $Q \subseteq V(G)$,
$$
\begin{array}{lll}
|Q|=1 & \Rightarrow &  \displaystyle{\frac{\dirp(Q)}{\pi(Q)}} = 1,\\
|Q|=2 & \Rightarrow &  \displaystyle{\frac{\dirp(Q)}{\pi(Q)}}  \geq \frac{1}{3},\\
|Q| \in \{3,4,5\} & \Rightarrow &  \displaystyle{\frac{\dirp(Q)}{\pi(Q)}}  \geq \frac{1}{5}.\\
\end{array}
$$
To prove that  $\iota_{_3}(G)=\frac{1}{5}$ let $\{Q_{_1},Q_{_2},Q_{_3}\}$ be a set of disjoint subsets for which the minimum is achieved
with $|Q_{_{1}}| \leq |Q_{_{2}}| \leq |Q_{_{3}}|$. Then by the previous claim it is clear that $|Q_{_{1}}| \not = 1$ and hence we either
have $|Q_{_{1}}|=|Q_{_{2}}|=2$ or we must have $|Q_{_{3}}| \leq 5$. Hence, again by the previous claim in either case the mean flow is  greater than
or equal to $\frac{1}{5}$.\\
 By a similar case study,  one can characterize $3$ different kinds
of partitions as follows
\begin{eqnarray*}
&&|Q_{_1}|=2, |Q_{_2}|=3, |Q_{_3}|=5,\\
{\rm or }&&|Q_{_1}|=2, |Q_{_2}|=4, |Q_{_3}|=4,\\
{\rm or }&&|Q_{_1}|=3, |Q_{_2}|=3, |Q_{_3}|=4,
\end{eqnarray*}
which shows that $\tilde{\iota}_{_3}(G)$ is achieved for the partition $\{B_{_i}|\ i=1,2,3\}$.
 }\end{exm}
\begin{exm}{\label{NONGEO2}
In this example we show that, by modifying Example~\ref{NONGEO}, we can construct a graph $G$ for which
$\tilde{\iota}_{_{n}}(G) > \iota_{_{n+1}}(G)$.\\
Let $G_{_n}=(V_{_n},E_{_n})$ be a symmetric graph, where
$V_{_n} \isdef \{u,x_{_i},y_{_i},z_{_i},w_{_i}|\ 1\leq i\leq n\}$. For
every $i\in{\cal I}_{_n}$, the induced graph on
$\{x_{_i},y_{_i},z_{_i},w_{_i}\}$ is a path of length $3$ and
the vertex $u$ is adjacent to all vertices $x_{_j}$, for all $j\in{\cal I}_{_n}$.
For $i\in{\cal I }_{_{n-1}}$, let
$A_{_i} \isdef \{x_{_i},y_{_i},z_{_i},w_{_i}\}$ and also let
$A_{_n} \isdef \{x_{_n},y_{_n}\}$ and $A_{_{n+1}} \isdef \{z_{_n},w_{_n}\}$. Then,
\[\iota_{_{n+1}}(G_{_n})\leq \frac{1}{n+1}\sum_{i=1}^{n+1} \frac{\dirp(A_{_i})}{\pi(A_{_i})}=\frac{1}{n+1}\left(\frac{n-1}{7}+\frac{1}{3}+\frac{2}{4}\right). \]
Now, for $i\in{\cal I }_{_{n-1}}$, let
$B_{_i} \isdef \{x_{_i},y_{_i},z_{_i},w_{_i}\}$ and also let
$B_{_n} \isdef \{u,x_{_n},y_{_n},z_{_n},w_{_n}\}$. Then, by a similar
argument as in Example~\ref{NONGEO}, one can prove that
\[\tilde{\iota}_{_n}(G_{_n})=\frac{1}{n}\sum_{i=1}^{n} \frac{\dirp(B_{_i})}{\pi(B_{_i})}=\frac{1}{n}\left(\frac{n-1}{7}+\frac{n-1}{n+7}\right).\]
It is clear that if $n$ is large enough, then $\iota_{_{n+1}}(G_{_n})\leq (\frac{n}{n+1}) \tilde{\iota}_{_n}(G_{_n}) < \tilde{\iota}_{_n}(G_{_n})$.
}\end{exm}
\section{Computational aspects}\label{COMPUTATION}
\subsection{A Federer-Fleming-type theorem}\label{COAREA}
Our basic aim in this section is to find a functional definition
through proving a Federer-Fleming-type theorem. This not only is quite important theoretically (e.g. see \cite{Rot85})
 and along with Examples~\ref{NONGEO} and \ref{NONGEO2} justifies the correctness of our generalization, but also can be assumed as our first step
to approximate isoperimetric constants using well-chosen test functions.\\
First, we define a couple of function spaces as follows.
\begin{defin}{
We define the space of {\it unit positive functions} as,
$${\cal U}_{_{\pi}}^{^{+}}(G) \isdef \left \{ f \ \ | \ \ f \in  {\cal F}_{_{\pi}}^+(G)  \
{\rm and}\  \|f\|_{_{1,\pi}}=1  \right \}.$$ Also, a class of
functions $\{f_{_{i}}\}^{^{n}}_{_{1}}$ is called {\it positive
orthonormal}, if for all $1 \leq i \leq n$ we have $f_{_{i}} \in
{\cal U}_{_{\pi}}^{^{+}}(G) $ and moreover, for all pairs of indices
$i \not = j$ we have $f_{_{i}} \bot_{_{\pi}} f_{_{j}}$. In this
regard we define
$$\distf^+_{_{n}}(G) \isdef \left \{ \{f_{_{i}}\}^{^{n}}_{_{1}} \ \ | \ \
\{f_{_{i}}\}^{^{n}}_{_{1}} \ {\rm is\ positive\ orthonormal}
\right \}.$$
$$\tilde{{\cal O}}^+_{_{n}}(G) \isdef \left \{ \{f_{_{i}}\}^{^{n}}_{_{1}} \in  \distf^+_{_{n}}(G) \ \ | \ \
\{{\rm supp}(f_{_{i}})\}_{_{1}}^{^{n}} \in {\cal P}_{_{n}}(G) \right
\}.$$
 }\end{defin}
Now,
given a kernel $K$ and a  nowherezero stationary distribution $\pi$,
let us consider the following parameters which are naturally related
to the constants $\iota_{_{n}}(G)$ and $\tilde{\iota}_{_{n}}(G)$,
$$\gamma_{_{n}}(G,K) \isdef \displaystyle{\inf_{_{
\{f_{_{i}}\}^{^{n}}_{_{1}} \in  {\cal O}^+_{_{n}}(G) }} } \
\frac{1}{n} \left ( \displaystyle{\sum^{n}_{i=1}}\ \| \dirg f_{_{i}}
\|_{_{1,\phi}} \right ).$$
$$\tilde{\gamma}_{_{n}}(G,K) \isdef \displaystyle{\inf_{_{
\{f_{_{i}}\}^{^{n}}_{_{1}} \in  \tilde{{\cal O}}^+_{_{n}}(G) }} }
\ \frac{1}{n} \left ( \displaystyle{\sum^{n}_{i=1}} \ \| \dirg
f_{_{i}} \|_{_{1,\phi}} \right ).$$
As usual, we exclude the kernel
when it is fixed or is clear from the context. First, we prove the following technical result.
\begin{lem}\label{FLOWGRAD}
For every function $f\in {\cal F}^+_{_{\pi}}(G)$ with
$\|f\|_{_{1,\pi}}=1$, there is a set $Q\subseteq {\rm supp}(f)$,
such that
 $\displaystyle{\frac{\dirp (Q)}{\pi(Q)}}\leq\|\dirg f\|_{_{1,\phi}}$.
\end{lem}
\begin{proof}{
We prove the claim by induction on the size of
the range of $f$, $Range(f)$. If $f$ takes only two values $0$ and $t_{_1}$,
 then $t_{_1}=\frac{1}{\pi(A)}$, where $A \isdef f^{-1}(t_{_1})$ and $f=t_{_1}\ \chi_{_A}$.
The proof is straightforward in this case.\\
Now, let $Range(f)=\{0,t_{_1},\ldots,t_{_n}\}$ such that $0< t_{_1}\leq\ldots\leq t_{_n}$ and, moreover,  for each $i\in{\cal I}_{_{n}}$
let $A_{_i} \isdef f^{-1}(t_{_i})$ and $\pi_{_i} \isdef \pi(A_{_i})$.
Then,
 \begin{eqnarray*}
 \|\dirg f\|_{_{1,\phi}}&=&\sum_{_{uv\in E(G)}} (f(u)-f(v))^+\
 \phi(u,v)\\
 &=&c_{_1} t_{_1}+c_{_2} t_{_2}+\ldots+c_{_n}t_{_n},
 \end{eqnarray*}
 for some real numbers $c_{_i}$, where each
 $c_{_i}$ depends only on flows between the subsets $A_{_i}$.
  Now, our objective is to find the minimum of the function
 $$\psi(x_{_1},\ldots,x_{_n}) \isdef c_{_1}x_{_1}+\ldots+c_{_n}x_{_n}$$
  subject to the constraints
 \begin{equation}\label{cons}\left\{\begin{array}{l}
 0\leq x_{_1}\leq\ldots\leq x_{_n},\\[2mm]
 g(x)\isdef\sum_{i=1}^n \pi_{_i} x_{_i}-1=0.
 \end{array}\right.\end{equation}
First, note that if $\frac{c_{_1}}{\pi_{_1}}=\ldots=\frac{c_{_n}}{ \pi_{_n}}=\kappa$, then
 $\psi(x_{_1},\ldots,x_{_n})=\sum_{i}\ \kappa\pi_{_i} x_{_i}=\kappa$ everywhere. Therefore,
\[\|\dirg f\|_{_{1,\phi}}= \kappa=\frac{c_{_n}}{\pi_{_n}}=\frac{\dirp(A_{_n})}{
 \pi(A_{_n})}.\]
 Now assume there exist $i \not = j$ such that $\frac{c_{_i}}{\pi_{_i}} \not = \frac{c_{_j}}{\pi_{_j}}$.
By Lagrange method, the the minimum of $\psi$ subject to the constraints (\ref{cons}) is equal to the minimum of the function
$h(x,\lambda)=\psi(x)-\lambda g(x)$ under constraints
\[ 0\leq x_{_1}\leq\ldots\leq x_{_n}.\]
  Since we have $\partial h/\partial x_{_i}=c_{_i}-\lambda\pi_{_i}$, by assumption partial derivatives are not simultaneously zero, and consequently,
   $g$ attains its minimum on a boundary point $(s_{_1},\ldots,s_{_n})$, i.e. there exists $i_{_0}$ such that $s_{_{i_0}}=s_{_{i_0+1}}$. Now, define the function  $\hat{f}$ to be equal to $s_{_i}$ on $A_{_i}$ for all $i$ and  zero elsewhere.
  Then, using (\ref{cons}), we have  $\|\hat{f}\|_{_{1,\pi}}=1$ and
  \[\|\dirg \hat{f}\|_{_{1,\phi}}=c_{_1}s_{_1}+\ldots+c_{_n} s_{_n}.\]
  Therefore, by induction hypothesis we can find a subset $Q\subseteq {\rm supp}(\hat{f})\subseteq {\rm
  supp}(f)$ such that
  \[\frac{\dirp (Q)}{
 \pi(Q)}\leq\|\dirg \hat{f}\|_{_{1,\phi}}=\psi(s_{_1},\ldots,s_{_n})\leq \psi(t_{_1},\ldots,t_{_n})=\|\dirg f\|_{_{1,\phi}}.\]
 }\end{proof}
The following theorem
presents a functional definition for the mean isoperimetric spectrum.
\begin{thm}\label{IOTAEQGAMMA}
For any graph $G$ $($and a given kernel $K$ on it$)$ and  for all\\
$1 \leq n \leq |V(G)|$ we have
$$\iota_{_{n}}(G) = \gamma_{_{n}}(G)=\tilde{\gamma}_{_n}(G).$$
\end{thm}
\begin{proof}{
By considering characteristic functions
 of sets we have $\gamma_{_{n}}(G)\leq \iota_{_{n}}(G)$. To prove
 equality, let $\{f_{_i}\}_{_1}^{^n}\in {\cal O}^+_{_n}(G)$ be chosen such that $\gamma_{_{n}}(G)$ is achieved.
 By Lemma~\ref{FLOWGRAD}, for every $i\in {\cal I}_{_n}$, there exists a set $Q_{_i}\subseteq {\rm supp}(f_{_i})$,
such that $\displaystyle{\frac{\dirp (Q_{_i})}{\pi(Q_{_i})}}\leq\|\dirg
f_{_i}\|_{_{1,\phi}}$. Since $\{f_{_i}\}_1^n$ is positive orthonormal and
$Q_{_i}\subseteq {\rm supp}(f_{_i})$, we conclude that $Q_{_i}$'s are disjoint subsets, and consequently,
\[\iota_{_{n}}(G)\leq \frac{1}{n}\sum_{i=1}^n \frac{\dirp (Q_{_i})}{ \pi(Q_{_i})}\leq \frac{1}{n}\sum_{i=1}^n \|\dirg
f_{_i}\|_{_{1,\phi}}=\gamma_{_n}(G).\]
For the second equality, by definition we have
$\iota_{_n}(G)=\gamma_{_n}(G)\leq\tilde{\gamma}_{_n}(G)$. Now, let
$\{Q_{_i}\}_{_1}^{^n}$ be a set of disjoint sets for which
$\iota_{_{n}}(G)$ is achieved. Also, let
$$ Q' \isdef \bigcup_{i=1}^{n-1} Q_{_i}, \quad  Q^{^*} \isdef V(G) -  Q',$$
 and
$0<\epsilon<\frac{1}{\pi(Q_{_n}\cup Q^{^*})}$ be an arbitrary fixed number.
For $i\in{\cal I}_{_{n-1}}$, define functions $\{g_{_i}\}_{_1}^{^n}\in
\tilde{{\cal O}}^+_{_{n}}(G)$ as $g_{_i} \isdef \frac{1}{\pi(Q_{_i})}\
\chi_{_{Q_{_i}}}$, and
\[
g_{_n}(u) \isdef \left\{\begin{array}{ll}\frac{1-\epsilon\ \pi(Q^{^*})}{
\pi(Q_{_n})}&\ u\in Q_{_n}\\\epsilon&\ u\in Q^{^*}\\0 &\ u\in
Q'.\end{array}\right.\] Therefore, we have
\begin{eqnarray*}
n\ \tilde{\gamma}_{_n}(G)&\leq&  \displaystyle{\sum^{n}_{i=1}} \
\| \dirg g_{_{i}} \|_{_{1,\phi}} \\&=&  \sum^{n-1}_{i=1}
\frac{\dirp(Q_{_i})}{\pi(Q_{_i})}+\frac{1-\epsilon\ \pi(Q^{^*})}{
\pi(Q_{_n})}\ \dirp(Q_{_n})+\epsilon \left(\dirp(Q^{^*},Q')-
\dirp(Q_{_n},Q^{^*})\right),
\end{eqnarray*}
and by tending $\epsilon$ to zero we get
\[\tilde{\gamma}_{_n}(G)\leq\frac{1}{n} \left( \sum^{n}_{i=1}
\frac{\dirp(Q_{_i})}{\pi(Q_{_i})}\right)=\iota_{_n}(G).\]
 }\end{proof}
This result along with Examples~\ref{NONGEO} and \ref{NONGEO2} show that the natural and correct definition of the mean isoperimetric constants
is what is defined in terms of minimization over {\it disjoint} subsets of the domain and not {\it partitions}.
\subsection{Spectral bounds and Cheeger-type inequalities}\label{SECGENCHEEGER}
In this section  we consider the problem of approximating the isoperimetric constants of a graph using its Laplacian spectrum
and some more information from the eigenspaces. As a by product of this, we also prove generalized versions of
Cheeger inequality for the isoperimetric
spectrum. To begin, let us recall an interesting
variational principle due to Ky Fan.
\begin{alphthm}\label{KYFAN}{\rm (e.g. see \cite{BAH97}) Ky Fan's minimum principle}\\
Let $A \in {\rm End(V)}$ be a self-adjoint matrix operating on the
$\nu$-dimensional inner-product space $V$, and let
$$\lambda_{_{1}} \leq \lambda_{_{2}} \leq \cdots \leq \lambda_{_{\nu}},$$
be the set of eigenvalues of $A$ ordered in an increasing order.
Then for any $1 \leq n \leq \nu$ we have
$$ \overline{\lambda}_{_{n}} =\frac{1}{n}\left ( \displaystyle{\min_{UU^*=id_{_{n}}}}\ \tr(UAU^*) \right ),$$
where $\overline{\lambda}_{_{n}}$ is the average of the $n$
smallest eigenvalues of $A$, $U$ is an arbitrary $n \times \nu$
matrix, $id_{_{n}}$ is the $n \times n$ identity matrix, and
$(\tr)$ is the trace function.
\end{alphthm}
Note that another way of expressing Ky Fan's result is that,
subject to the same conditions of Theorem~\ref{KYFAN},
$$\overline{\lambda}_{_{n}}=\frac{1}{n}\left ( \displaystyle{\min_{\genfrac{}{}{0pt}{}{\ f_{_{i}} \bot f_{_{j}}}{f_{_{i}} \not = 0}} \sum^{n}_{i=1}}\
\frac{\langle A f_{_{i}},f_{_{i}}\rangle}{\| f_{_{i}} \|^{^2}_{_{2}}} \right
).$$
The second important fact is related to  the concept of a {\it nodal
domain} (e.g. see \cite{BLS07} and references therein). It is interesting to note that in the
continuous case, the eigenfunctions of the ordinary Laplacian (say
of a compact Riemannian manifold) is always a continuous function
(essentially smooth) and by Rolle's theorem there is always a zero
point between any two points with different signs. This fact, in a
way, justifies the study of connected components of $f^{-1}(0)$ (as
{\it nodal regions} \cite{DETAL,FRI,JNT01}) for any eigenfunction
$f$ in the continuous case. However, when we are dealing with a
discontinuous object as a graph, an eigenfunction can have opposite
signs on the two endpoints of an edge, where this, on the one hand,
makes the whole thing more complex, and on the other hand, it makes
the space of eigenfunctions far richer.
\begin{defin}{
If $f \in {\cal F}_{_{\pi}}(G)$ and $Q \subseteq V(G)$, the pair
$(Q,Q^c)$ is called a {\it bipolar cut-set} for $f$ if for any edge
 $uv \in \sb(Q)$ we have $f(u)f(v) \leq 0$. Also, a subset $Q$  is
called a {\it nonnegative} ({\it nonpositive}) {\it bipolar part} of
$f$ if
 $f_{_{1}} \isdef f|_{_{Q}}$ is a nonnegative (nonpositive) function on
$Q$  and $(Q,Q^c)$ is a bipolar cut-set for $f$. A {\it signed} part
of $f$ is a subset $Q$ that is either  a nonnegative or a
nonpositive  bipolar part of $f$. Note that in this case
$f=f_{_{1}}+f_{_{2}}$ where $f_{_{2}} \isdef f|_{_{Q^c}}$ and
$f_{_{1}} \bot f_{_{2}}$.
Also, note that any {\it strong sign-graph} of $f$ is clearly a signed part of $f$
(e.g. see \cite{BLS07} for the definitions, other variations and background).\\
 For a given real number $\zeta \in {\mathbb R}$, a real function  $f \in {\cal
F}(G)$ is said to be {\it $\zeta$-excessive} (resp. {\it
$\zeta$-deficient}) for the kernel $K$ if $K f \leq \zeta f$ (resp.
$K f \geq \zeta f$). By abuse of language, a $\zeta$-excessive
(resp. $\zeta$-deficient) function for $\siml$ is just referred to
as a $\zeta$-excessive (resp. $\zeta$-deficient) function, if
details are clear from the context.
 }\end{defin}
We will use the following lemma to prove a generalized Cheeger
inequality later. It is instructive to mention  that the lemma can also be deduced as a corollary of the well-known
Duval-Reiner lemma (e.g. see \cite{DR99} for the lemma and \cite{BLS07} for the history, erratum and a more detailed discussion).
\begin{lem}
\label{DRCOR} Let $G$ be graph, and $f \in {\cal F}(G)$ be a
$\zeta$-excessive $($resp. $\zeta$-deficient$)$ function  for
$\siml$, such that a subset $Q \subseteq V(G)$ is a nonnegative
$($resp. nonpositive$)$ bipolar part of $f$. Then, assuming $g
\isdef f|_{_{Q}}$ we have
$$\zeta \geq \frac{\|\ssimg g \|^{^2}_{_{2,\overline{\phi}}}}{\| g \|^{^2}_{_{2,\pi}}}.$$
\end{lem}
\begin{proof}{Let $Q$ be a nonnegative bipolar part of the $\zeta$-excessive function $f$.
If $u\in Q$ then $g(u)=f(u)$ and $g(v)\geq f(v)$ for all
neighbors $v$ of $u$, and so $\siml g(u)\leq \siml f(u)\leq\zeta
f(u)=\zeta g(u)$. Also if $u\not\in Q$ then $g(u)=0$, and since
$g\geq 0$, trivially $\siml g(u)\leq \zeta g(u)$. Hence by
Lemma~\ref{LEMGER}(c) we have
\[\|\ssimg g
\|^{^2}_{_{2,\overline{\phi}}}=\langle \siml g,
g\rangle_{_\pi}\leq \zeta \langle g,g\rangle_{_\pi}=\zeta \| g
\|^{^2}_{_{2,\pi}}.\]
 }\end{proof}
The following is the first half of the generalized Cheeger
inequality.
\begin{thm}\label{CHEEGER1}
For any given graph $G$ we have $\overline{\lambda}_{_{n}} \leq
\iota_{_{n}}(G).$
\end{thm}
\begin{proof}{
Let $\{Q_{_{i}}\}^{^{n}}_{_{1}} \in  {\cal D}_{_{n}}(G)$ be chosen
such that
$$\iota_{_{n}}(G)= \frac{1}{n}
\displaystyle{\sum^{n}_{i=1}} \frac{\dirp
(Q_{_{i}})}{\pi(Q_{_{i}})}.$$
 For every $i \in {\cal I}_{_{n}}$ define
$$h_{_{i}}(u)=\chi_{_{Q_{_{i}}}} \isdef \left \{ \begin{array}{ll}
 1 & u \in Q_{_{i}}  \\
 0 & u \not \in Q_{_{i}}. \end{array}\right. $$
Now, for each $i \in {\cal I}_{_{n}}$ we have $\| h_{_{i}} \|^{^2}_{_{2,\pi}}=\pi(Q_{_{i}})$ and by Lemma~\ref{LEMGER},
$$\| \ssimg h_{_{i}} \|^{^2}_{_{2,\overline{\phi}}}=\| \ssimg h_{_{i}} \|_{_{1,\overline{\phi}}}= \| \dirg h_{_{i}}\|_{_{1,\phi}}= \dirp(Q_{_{i}}).$$
 Hence, by Ky Fan's minimum principle,
$$n\ \overline{\lambda}_{_{n}}=\displaystyle{\min_{f_{_{i}} \bot_{_{\pi}} f_{_{j}}} \sum^{n}_{i=1}}\
\frac{\langle\siml f_{_{i}},f_{_{i}}\rangle_{_{\pi}}}{\| f_{_{i}}
\|^{^2}_{_{2,\pi}}} \leq  \displaystyle{\sum^{n}_{i=1}}\ \frac{\| \ssimg h_{_{i}} \|^{^2}_{_{2,\overline{\phi}}}}{\| h_{_{i}} \|^{^2}_{_{2,\pi}}}
= n\ \iota_{_{n}}(G).$$
 }\end{proof}
Note that for the complete graph ${\sf K}_{_t}$, we have $\lambda_{_i}=\frac{t}{t-1}$,
for all $2\leq i\leq t$. Also, for the complete bipartite graph ${\sf
K}_{r,s}$, we have $\lambda_{_i}=1$, for all $2\leq i\leq r+s-1$ and
$\lambda_{_{r+s}}=2$. Thus, by Example~\ref{COMPLETE}, for the complete
graph ${\sf K}_{_t}$, when $1\leq n\leq t$ and also, for the complete
bipartite graph ${\sf K}_{_{r,s}}$, when $s$ is a multiple of $r$
and $1\leq n\leq r$, equality holds in Theorem~\ref{CHEEGER1}, which shows that the result is sharp.\\
For the second half of a generalized Cheeger inequality we need the following definition.
\begin{defin}{
Let $\Gamma = (\zeta_{_{1}},\zeta_{_{2}},\ldots ,\zeta_{_{n}})$ be
an $n$-list of real numbers. Then an $n$-list of real functions $F=
(f_{_{1}},f_{_{2}},\ldots ,f_{_{n}})$ on a domain $X$ along with $n$
disjoint subsets ${\cal Q} = (Q_{_{1}},Q_{_{2}},\ldots ,Q_{_{n}})$
such that $Q_{_{i}} \subseteq X$, is called a {\it compatible transverse set of functions} for $\siml$, if:
\begin{itemize}
\item{$f_{_{i}}|_{_{Q_{_{i}}}} \not = 0.$}
\item{For each $1 \leq i \leq n$, the function
$f_{_{i}}$ is a $\zeta_{_{i}}$-excessive  (resp. $\zeta$-deficient) function (with respect to $\siml$) on $X$.}
\item{For each $1 \leq i \leq n$, the subset $Q_{_{i}}$ is a nonnegative (resp. nonpositive)
 bipolar part of $f_{_{i}}$.}
\end{itemize}
 }\end{defin}
\begin{thm}\label{CHEEGER2}
Consider a graph $G$ and let $\Gamma
=(\zeta_{_{1}},\zeta_{_{2}},\ldots ,\zeta_{_{n}})$. If
$F=(f_{_{1}},f_{_{2}},\ldots ,f_{_{n}})$ along with ${\cal
Q}=(Q_{_{1}},Q_{_{2}},\ldots ,Q_{_{n}})$ is a compatible transverse
set of functions for $\siml$, then
 $$2\ \overline{\zeta}_{_{n}}\   \geq \  \iota_{_{n}}(G)^2.$$
\end{thm}
\begin{proof}{
Let $0 \not = g_{_{i}} \isdef f_{_{i}}|_{_{Q_{_{i}}}}$. Then,
$$
\begin{array}{rl}
\overline{\zeta}_{_{n}} &          \geq
 \frac{1}{n}\   {
                         \displaystyle{\sum^{n}_{i=1}}\
                         \frac{\| \ssimg g_{_{i}} \|^{^2}_{_{2,\overline{\phi}}}}{\| g_{_{i}}
                         \|^{^2}_{_{2,\pi}}}  }
                    \geq  \frac{1}{2n}\  { \displaystyle{\sum^{n}_{i=1}}\
                         \frac{\| \dirg g^2_{_{i}} \|^{^2}_{_{1,\phi}}}{\| g^2_{_{i}}
                         \|^{^2}_{_{1,\pi}}} }\\
                           & \\
                         & \geq  \frac{1}{2}\ \left ( \frac{1}{n}\ { \displaystyle{\sum^{n}_{i=1}}\
                         \frac{\| \dirg g^2_{_{i}} \|_{_{1,\phi}}}{\| g^2_{_{i}}
                         \|_{_{1,\pi}}} } \right )^2    \geq  \frac{1}{2}\  \iota_{_{n}}(G)^2,\\
\end{array}
$$
where the first and the second inequalities follow from
Lemma~\ref{DRCOR} and \ref{LEMNORM2}, respectively, and the
third one is a direct application of Cauchy-Schwarz inequality.
 }\end{proof}
It ought to be noted that Theorems~\ref{CHEEGER1} and \ref{CHEEGER2} together, can be considered as a {\it generalized} Cheeger inequality.
In what follows we deduce a special case where one may get an explicit inequality for the mean spectrum.
\begin{thm}\label{GENCHEEGER}
Consider a kernel $K$ on a base graph $G$. Let
$F=(f_{_{2}},f_{_{3}},\ldots ,f_{_{n+1}})$ be a list of
eigenfunctions of $\siml$ for the list of eigenvalues $\Gamma
=(\lambda_{_{2}},\lambda_{_{3}},\ldots ,\lambda_{_{n+1}})$,
respectively, such that along with  ${\cal
Q}=(Q_{_{2}},Q_{_{3}},\ldots ,Q_{_{n+1}})$ form a compatible
transverse set of functions for $\siml$.
 Then,
\begin{equation}
\overline{\lambda}_{_{n}} \leq
\iota_{_{n}}(G) \leq \sqrt{\frac{2(n+1)}{n}\
\overline{\lambda}_{_{n+1}}}.
\end{equation}
\end{thm}
Moreover, we would like to add that
following the same scenario described for the mean version, one may define the $n$th max-isoperimetric constant as
$$\varsigma_{_n}(K,\pi) \isdef \displaystyle{\min_{_{ \{Q_{_{i}}\}^{^{n}}_{_{1}} \in  {\cal
D}_{_{n}}(G) }} } \ \left ( \displaystyle{\max_{_1\leq i\leq n}}\
\frac{\dirp (Q_{_{i}})}{\pi(Q_{_{i}})} \right ).$$
It is noteworthy that all of the previous mentioned results such as the
Federer-Fleming theorem can also be verified for this version with appropriate modifications.
For instance, we may state a more standard
Cheeger inequality  for the max-isoperimetric constant $\varsigma_{_n}$ using Theorems~\ref{CHEEGER1} and \ref{CHEEGER2} and their counterparts, along with  Courant-Fischer variational theorem as follows.
\begin{thm}\label{CORCHEEGER2}
For a given graph $G$, let $f$ be an eigenfunction of $\siml$
corresponding to the $n$th eigenvalue $\lambda_{_n}$. Also, let
$(Q_{_{1}},Q_{_{2}},\ldots ,Q_{_{n}})$ be a list of $n$ disjoint nonempty
subsets of $V(G)$ such that for every $1 \leq i \leq n$ we have $f|_{_{Q_{_{i}}}} \not = 0$ and each $Q_{_{i}}$ is a nonnegative or
nonpositive bipolar part of $f$. Then,
\begin{equation}\label{CHEEGEREQ1}
\frac{\lambda_{_n}}{2} \leq \varsigma_{_n}(G)\leq \sqrt{2\ \lambda_{_n}}, \quad {\rm and} \quad  \overline{\lambda}_{_n} \leq \iota_{_n}(G)\leq \sqrt{2\ \lambda_{_n}}.
\end{equation}
\end{thm}
Also, as a corollary of Theorem~\ref{CORCHEEGER2} by considering the fact that always the second eigenvalue has an eigenfunction with two nodal domains, we obtain the classical Cheeger inequality as,
\begin{equation}\label{CHEEGERCLASSIC1}
\frac{\lambda_{_{2}}}{2} \leq\iota_{_2}(G)\leq \varsigma_{_{2}}(G) \leq \sqrt{2\
\lambda_{_{2}}}.
\end{equation}
It also must be emphasized that a direct use of eigenvalues and
eigenfunctions (not necessarily tuned with repetition) in
Theorem~\ref{CHEEGER2} will definitely
make a deviation from sharpness which can be easily verified by a
comparison to the classical Cheeger inequality
(Inequality~(\ref{CHEEGERCLASSIC1})). Note that the classical
Cheeger inequality is far from being sharp by a recent
result of Montenegro and Tetali~\cite{MT06}.\\
To provide some examples let us recall the following result.
\begin{alphthm}{\rm \cite{BIY03,BLS07}}
Let $K$ be a kernel on a tree $T$ and let $f_{_{n}}$ be an eigenfunction of $\siml$ with eigenvalue $\lambda_{_n}$ which does not
vanish on any vertex. Then $\lambda_{_n}$ is simple and  $f_{_{n}}$ has exactly $n$ strong nodal domains.
\end{alphthm}
Therefore, a generalized Cheeger inequality  is valid for any Markov chain on a tree $T$ with a nowherezero eigenfunction
$f_{_{n}}$ of an eigenvalue $\lambda_{_n}$, i.e.
$$\min (\overline{\lambda}_{_n},\frac{\lambda_{_n}}{2}) \leq \iota_{_{n}}(T)  \leq  \varsigma_{_n}(T)\leq \sqrt{2\ \lambda_{_n}}.$$
For more on the extensive literature of Markov chains on trees the interested reader is referred to
\cite{BLS07,MIC08} and references therein.\\
On the other hand, it is quite interesting that even for the case of trees we do not know enough about the behavior of parameters discussed
in this article, and as Example~\ref{NONGEO} shows one encounters nongeometric trees in very small cases. Hence, we believe that the following
problem can be considered to be a nice starting point for the study of supergeometric graphs.\\
\begin{prb}
Characterize the class of  supergeometric trees.
\end{prb}
\subsection{Algorithmic considerations}\label{KMEANS}
In this section we touch on some algorithmic aspects of the isoperimetry problem and we study its relationships to some well-known concepts as the $k$-means algorithm and the normalized cuts method. This section is mainly influenced by the seminal contribution of J.~Malik and J.~Shi \cite{SIMA00}
(also see \cite{DGK04})
that was brought to our attention after the presentation of the first two authors' article on the isoperimetric spectrum of graphs \cite{DAHAARXIV}.\\
Following our notations in Section~\ref{ISOSPEC},  for a set $X$, ${\cal D}_{_n}(X)$
stands for the set of all $n$-sets $\{Q_{_i}\}_1^n$, where $Q_{_i}$'s are nonempty disjoint subsets of $X$. Also ${\cal P}_{_n}(X) \subseteq {\cal D}_{_n}(X)$ consists of all $n$-partitions of $X$.
\begin{defin}{ Given a  function $f\in {\cal F}^d(X)$ and a weight function $\omega:X\to \mathbb{R}^+-\{0\}$,
for every $1\leq n\leq |X|$, the cost function
${\cal C}^{^{f,\omega}}_{_n}: {\cal D}_{_n}(X) \to \mathbb{R}^+$ is
defined as follows
\[{\cal M}^{^{f,\omega}}_{_n}(\{Q_{_i}\}_1^n)\isdef\sum_{i=1}^n\sum_{u\in Q_{_i}} \omega(u)\|f(u)-\mathbf{m}_{_i}\|^2,
\mbox{ where }\mathbf{m}_{_i}\isdef\frac{\sum_{u\in
Q_{_i}}\omega(u)f(u)}{\sum_{u\in Q_{_i}}\omega(u)},\]
 and
\[{\cal C}^{^{f,\omega}}_{_n}(\{Q_{_i}\}_1^n) \isdef  {\cal M}^{^{f,\omega}}_{_n}(\{Q_{_i}\}_1^n)+\sum_{u\in Q^{^*}} \omega(u)\|f(u)\|^2,\]
where  $\|.\|$ is the Euclidean $L^2$-norm of $\mathbb{R}^d$ and $Q^{^*}=X - \cup_{i=1}^n Q_{_i}$.
Also, associated to the functions $f$ and $\omega$, we define the following
parameters
$${\mu}_{_n}(f,\omega) \isdef \min_{{\cal Q}\in {\cal D}_{_n}(X)}\ {\cal C}^{^{f,\omega}}_{_n}({\cal Q}), \quad {\rm and} \quad
\tilde{\mu}_{_n}(f,\omega)\isdef \min_{{\cal Q}\in {\cal
P}_{_n}(X)}\ {\cal C}^{^{f,\omega}}_{_n}({\cal Q}).$$
 }\end{defin}
The well-known $k$-means algorithm seeks for the value of $\tilde{\mu}_{_n}(f,\omega)$ and an $n$-partition on which the minimum is achieved.
First, let's state the following simple lemma.
\begin{lem}\label{LEMKMEANS}
Given functions $f\in {\cal
F}^d(X)$ and  $\omega:X\to \mathbb{R}^+-\{0\}$  on $X$, for all $1\leq n\leq |X|$ and for every ${\cal Q}=\{Q_{_i}\}_1^n\in {\cal D}_{_n}(X)$, we have
\[{\cal M}^{^{f,\omega}}_{_n}({\cal Q})=\sum_{i=1}^n\frac{1}{2\ \omega(Q_{_i})}\sum_{u,v\in Q_{_i}} \omega(u)\omega(v)\|f(u)-f(v)\|^2.\]
\end{lem}
\begin{proof}{
\begin{eqnarray*}
{\cal M}^{^{f,\omega}}_{_n}({\cal Q})&=&
\sum_{i=1}^n\frac{1}{\omega(Q_{_i})^2}\sum_{u\in Q_{_i}}
\omega(u)\left \|\sum_{v\in Q_{_i}}(f(u)-f(v))\ \omega(v)\right
\|^2\\&=&\sum_{i=1}^n\frac{1}{\omega(Q_{_i})^2}\sum_{u,v,w\in
Q_{_i}} \omega(u)\omega(v)\omega(w)\ \Big( \langle
f(u),f(u)\rangle-\langle f(v),f(u)\rangle\\&&\hspace{-3mm}-\langle
f(w),f(u)\rangle+\langle
f(v),f(w)\rangle\Big)=\sum_{i=1}^n\frac{1}{2\
\omega(Q_{_i})}\sum_{u,v\in Q_{_i}}
\omega(u)\omega(v)\|f(u)-f(v)\|^2.
\end{eqnarray*}
}\end{proof}
Let $G$ be a graph on $\nu$ vertices  and
$K$ be a kernel on it, together with a stationary
distribution $\pi$. Also, let $D$ be the diagonal matrix defined as $D(u,u) \isdef \pi(u)$. Define, the {\it normalized flow matrix} as,
$${\Phi}\isdef(id+\overline{K})D^{-1},$$
whose $(u,v)$ entry can be described as
$$\Phi(u,v)= \left \{ \begin{array}{ll}
\frac{\overline{\phi}(u,v)}{\pi(u)\pi(v)} & u \not = v,  \\
&\\
\frac{1}{\pi(u)}+\frac{\overline{\phi}(u,u)}{\pi(u)^2}  & u =v, \end{array}\right. $$
 that justifies the name.  Now, if $f$ is an eigenfunction for the eigenvalue $\lambda$ of $\Phi$, and we chose $x$ in such a way that
 $\frac{f(x)}{\pi(x)}=\max_u\frac{|f(u)|}{\pi(u)}$, then
$$\lambda f(x)=(\Phi f)(x)=\frac{f(x)}{\pi(x)}+\sum_u \frac{\overline{K}(x,u)}{\pi(u)} f(u)
\geq\frac{f(x)}{\pi(x)}\left(1-\sum_u\overline{K}(x,u)\right)=0,$$
which shows that  $\Phi$ is a positive semidefinite matrix, and consequently,  there exists a
matrix $P$ such that $\Phi=P^tP$. Let us define the function
$p_{_K}\in{\cal F}^{\nu}(G)$ such that $p_{_K}(u)$ is $u$th
column of $P$.
\begin{pro}
For every graph $G$ on $\nu$ vertices and a kernel $K$ on it with a nowherezero stationary
distribution $\pi$, the following equations hold for all $1\leq n\leq \nu$,
\begin{eqnarray}
{\mu}_{_n}(p_{_K},\pi)=\nu-2n+\tr(K)+n\
{\iota}_{_n}(G),\label{KM-ISO}\\[2mm]
\tilde{\mu}_{_n}(p_{_K},\pi)=\nu-2n+\tr(K)+n\
\tilde{\iota}_{_n}(G).\nonumber
\end{eqnarray}
\end{pro}
\begin{proof}{
Let ${\cal Q}=\{Q_{_i}\}_1^n\in {\cal D}_{_n}(G)$ be chosen
arbitrary and let $Q^{^*}=V(G)- \cup_{i=1}^n Q_{_i}$. By
Lemma~\ref{LEMKMEANS},
\begin{eqnarray*}
{\cal C}^{^{p_K,\pi}}_{_n}({\cal Q})&=&
\sum_{i=1}^n\frac{1}{2\ \pi(Q_i)}\sum_{u,v\in Q_{_i}} \pi(u)\pi(v)\|p_{_K}(u)-p_{_K}(v)\|^2+\sum_{u\in Q^{^*}}\pi(u)\|p_{_K}(u)\|^2\\
&=&\sum_{u\in V(G)}\pi(u)\langle
p_{_K}(u),p_{_K}(u)\rangle-\sum_{i=1}^n\sum_{u,v\in
Q_{_i}}\frac{\pi(u)\pi(v)\langle
p_{_K}(u),p_{_K}(v)\rangle}{\pi(Q_{_i})}\\
&=&\sum_{u\in V(G)}(1+K(u,u))-\sum_{i=1}^n\sum_{u\in
Q_{_i}}\frac{\pi(u)}{\pi(Q_{_i})}-\sum_{i=1}^n\sum_{u,v\in
Q_{_i}}\frac{\pi(u)\overline{K}(u,v)}{\pi(Q_{_i})}\\
&=&\nu+\sum_{u\in V(G)} K(u,u)-n-\sum_{i=1}^n\frac{\pi(Q_{_i})-\dirp(Q_{_i})}{\pi(Q_{_i})}\\
&=&\nu-2n+\tr(K)+\sum_{i=1}^n\frac{\dirp(Q_{_i})}{\pi(Q_{_i})}.
\end{eqnarray*}
Now, the equations follow by taking minimum  over ${\cal D}_{_n}(G)$
and  ${\cal P}_{_n}(G)$.
 }\end{proof}
This results along with
Theorem~\ref{IOTAEQGAMMA} shows that the target of the standard
$k$-means algorithm is not theoretically well-justified and must
be redefined to be the set of {\it disjoint} subsets for which
the minimum of $C^{^{f,\omega}}_{_n}$ is achieved. Besides,
note that the left side of Equation~(\ref{KM-ISO}) is always
nonnegative, and consequently, one finds a lower bound for
$\iota_{_n}(G)$ as follows, which is good when $n$ is large.
\begin{cor}
Let $G$ be a graph on $\nu$ vertices and a kernel  $K$ on it with a nowherezero  stationary
distribution $\pi$. Then for every integer $1\leq n\leq \nu$ we have
\[{\iota}_{_n}(G)\geq 2-\frac{\nu+\tr(K)}{n}.\]
\end{cor}
It should be noted that the set of functions
$\{p_{_K}(u) \ \ | \ \ u \in V(G)\}$ constitutes an orthogonal
representation  for $G$ \cite{LO79}. These relations along with
relationships of the subject to the theory of weakly unitary
invariant norms (e.g. see \cite{BAH97}), convex analysis on
Hermitian matrices, and
 applications of semidefinite programming to the
 approximation problem are among areas that ought to be considered  in forthcoming research.

\vspace*{.5cm}\ \\
\noindent
 {\bf Acknowledgement} \ \\ \ \\
The authors wish to express their sincere thanks to P.~Diaconis,
A.~Mehrabian, L.~Saloff-Coste, M.~Shahshahani and an anonymous referee
for their encouragement and many useful comments. They also would like to thank
L.~Miclo who drew their attention to his article \cite{MIC08}.
\ \\ \\

\end{document}